\documentclass[11pt]{article}

\usepackage[T1]{fontenc}
\usepackage[latin9]{inputenc}
\usepackage{geometry}
\usepackage{amsmath,amsthm,amssymb}
\usepackage{verbatim}
\usepackage{thmtools, thm-restate}

\usepackage{bbm}
\usepackage{bm}
\usepackage{color}

\usepackage{paralist}
\usepackage{csquotes}
\usepackage[unicode=true,pdfusetitle,
 bookmarks=true,bookmarksnumbered=false,bookmarksopen=false,
 breaklinks=false,pdfborder={0 0 0},pdfborderstyle={},backref=false,colorlinks=false]
 {hyperref}
 
\usepackage[square,sort,comma,numbers]{natbib} 

\makeatletter

\usepackage[nameinlink,capitalise,noabbrev]{cleveref}

\hypersetup{%
    bookmarksnumbered, bookmarksopen=true, bookmarksopenlevel=1,%
}

\theoremstyle{plain}
\newtheorem{thm}{Theorem}[section]
\crefname{thm}{Theorem}{Theorems}
\theoremstyle{plain}
\newtheorem*{lem*}{Lemma}
\newtheorem{lem}[thm]{Lemma}
\crefname{lem}{Lemma}{Lemmas}
\theoremstyle{plain}

\theoremstyle{plain}
\newtheorem*{claim*}{Claim}
\crefname{claim}{Claim}{Claims}
\theoremstyle{definition}
\newtheorem{defn}[thm]{Definition}
\theoremstyle{plain}
\newtheorem{conjecture}[thm]{Conjecture}
\newtheorem{prop}[thm]{Proposition}
\crefname{prop}{Proposition}{Propositions}
\theoremstyle{definition}

\theoremstyle{definition}

\theoremstyle{plain}

\newtheorem{ques}[thm]{Question}

\crefformat{equation}{#2(#1)#3}
\crefname{appsec}{Appendix}{Appendices}

\date{}

\usepackage{appendix}

\crefformat{equation}{#2(#1)#3}

\let\originalleft\left
\let\originalright\right
\renewcommand{\left}{\mathopen{}\mathclose\bgroup\originalleft}
\renewcommand{\right}{\aftergroup\egroup\originalright}
\usepackage{verbatim}

\makeatletter
\renewcommand*{\UrlTildeSpecial}{%
  \do\~{%
    \mbox{%
      \fontfamily{ptm}\selectfont
      \textasciitilde
    }%
  }%
}%
\let\Url@force@Tilde\UrlTildeSpecial
\makeatother

\makeatother
\begin{document}

\title{The smallest singular value of random combinatorial matrices}

\author{Tuan Tran\thanks{Discrete Mathematics Group, Institute for Basic Science (IBS), 34126 Daejeon, Republic of Korea. Email:
\href{mailto:tuantran@ibs.re.kr} {\nolinkurl{tuantran@ibs.re.kr}}.
This work was supported by the Institute for Basic Science (IBS-R029-Y1).}}

\maketitle
\global\long\def\bC{\mathbb{C}}
\global\long\def\bE{\mathbb{E}}
\global\long\def\bN{\mathbb{N}}
\global\long\def\Pr{\mathbb{P}}
\global\long\def\bR{\mathbb{R}}
\global\long\def\bS{\mathbb{S}}
\global\long\def\cS{\mathcal{S}}
\global\long\def\bZ{\mathbb{Z}}
\global\long\def\cB{\mathcal{B}}
\global\long\def\calD{\mathcal{D}}
\global\long\def\cE{\mathcal{E}}
\global\long\def\cF{\mathcal{F}}
\global\long\def\cH{\mathcal{H}}
\global\long\def\cQ{\mathcal{Q}}
\global\long\def\cL{\mathcal{L}}
\global\long\def\cY{\mathcal{Y}}
\global\long\def\cN{\mathcal{N}}
\global\long\def\cM{\mathcal{M}}
\global\long\def\tL{\widetilde{L}}
\global\long\def\tsigma{\widetilde{\sigma}}
\global\long\def\tu{\widetilde{u}}
\global\long\def\tbu{\widetilde{\bm{u}}}
\global\long\def\tw{\widetilde{\bm{w}}}
\global\long\def\CLCD{\mathrm{CLCD}}
\global\long\def\Comp{\mathrm{Comp}}
\global\long\def\calC{\mathrm{Cons}}
\global\long\def\La{\mathrm{\Lambda}}
\global\long\def\dist{\mathrm{dist}}
\global\long\def\Incomp{\mathrm{Incomp}}
\global\long\def\Sparse{\mathrm{Sparse}}
\global\long\def\Var{\mathrm{Var}}
\global\long\def\BL{\mathrm{BL}}
\global\long\def\Per{\mathrm{Per}}
\global\long\def\Ber{\mathrm{Ber}}
\global\long\def\Proj{\mathrm{Proj}}
\global\long\def\op{\mathrm{op}}
\global\long\def\HS{\mathrm{HS}}
\global\long\def\supp{\mathrm{supp}}
\global\long\def\spn{\mathrm{span}}
\global\long\def\sgn{\mathrm{sgn}}
\global\long\def\one{\boldsymbol{1}}
\global\long\def\range#1{\left[#1\right]}

\global\long\def\floor#1{\left\lfloor #1\right\rfloor }
\global\long\def\ceil#1{\left\lceil #1\right\rceil }
\global\long\def\product#1{\left\langle #1\right\rangle}
\global\long\def\al{\alpha}
\global\long\def\be{\beta}
\global\long\def\ga{\gamma}
\global\long\def\D{\mathrm{D}}
\global\long\def\de{\delta}
\global\long\def\eps{\varepsilon}
\global\long\def\la{\lambda}
\global\long\def\vp{\varphi}
\newcommand{\norm}[1]{\left| #1 \right|}
\newcommand\restr[2]{{
		\left.\kern-\nulldelimiterspace 
		#1 
		\vphantom{\big|} 
		\right|_{#2} 
}}

\begin{abstract}
Let $Q_n$ be a random $n\times n$ matrix with entries in $\{0,1\}$ whose rows are independent vectors of exactly $n/2$ zero components. We show that the smallest singular value $s_n(Q_n)$ of $Q_n$ satisfies
\[
\Pr\Big\{s_n(Q_n)\le \frac{\eps}{\sqrt{n}}\Big\} \le C\eps + 2 e^{-cn} \quad \text{for every} \enskip \eps \ge 0,
\]
which is optimal up to the constants $C,c>0$. This implies earlier results of Ferber, Jain, Luh and Samotij \cite{FJLS20} as well as Jain \cite{Jain20}. In particular, for $\eps=0$, we obtain the first exponential bound in dimension for the singularity probability:
\[
\Pr\big\{Q_n \enskip \text{is singular}\big\} \le 2 e^{-cn}.
\]

To overcome the lack of independence between entries of $Q_n$, we introduce an arithmetic-combinatorial invariant of a pair of vectors, which we call {\em a Combinatorial Least Common Denominator} (CLCD). We prove a small ball probability inequality for the combinatorial statistic $\sum_{i=1}^{n}a_iv_{\sigma(i)}$ in terms of the CLCD of the pair $(\bm{a},\bm{v})$, where $\sigma$ is a uniformly random permutation of $\{1,2,\ldots,n\}$ and $\bm{a}:=(a_1,\ldots,a_n), \bm{v}:=(v_1,\ldots,v_n)$ are real vectors. This inequality allows us to derive strong anti-concentration properties for the distance between a fixed row of $Q_n$ and the linear space spanned by the remaining rows, and prove the main result.
\end{abstract}

\section{Introduction}

Given a random $n\times n$ matrix $A$, the basic question is: {\em how likely is $A$ to be invertible,} and, more quantitatively, {\em well conditioned}? These questions can be expressed in terms of the singular values $s_k(A)$ of $A$, which are defined as the eigenvalues of $\sqrt{A^{\intercal} A}$ arranged in non-decreasing order. Of particular significance are the largest and smallest singular values, which admit the following variational characterization:
\[
s_1(A)=\max_{\bm{v \in \bS^{n-1}}}\norm{A\bm{v}} \quad \text{and} \quad  s_n(A)=\min_{\bm{v}\in \bS^{n-1}}\norm{A\bm{v}},
\]
where $\norm{\cdot}$ denotes the Euclidean norm on $\bR^{n}$, and $\bS^{n-1}$ is the unit Euclidean sphere in $\bR^{n}$.

\subsection{Random matrices with independent entries}

The behavior of the smallest singular values of random matrices with independent entries have been intensively studied \cite{BVW10,KKS95,LPRT05,LR12,LT20,Livshyts18,LTV19,MP14,RT18,Rudelson05,RV08,RV09,TV06,TV07,TV09,TV10b,Tatarko,Tikhomirov15,Tikhomirov16,Tikhomirov17,Tikhomirov20,Vershynin11}, in part due to applications in computer science and engineering. For random matrices with independent $N(0,1)$ Gaussian entries, Edelman \cite{Edelman88} and Szarek \cite{Szarek91} showed 

\begin{equation}\label{Spielman-Teng gaussian}
\Pr\Big\{s_n(A)\le \frac{\eps}{\sqrt{n}}\Big\} \sim \eps,
\end{equation}

thereby verifying a conjecture of Smale \cite{Smale85} as well as a speculation of von Neumann and Goldstine \cite{NG47}. Edelman's proof relied heavily on the rotation invariance property of the Gaussian distribution. Extending Edelman's theorem to general distributions is non-trivial. Spielman and Teng \cite{ST04} conjectured that \eqref{Spielman-Teng gaussian} should hold for random sign matrices, up to an exponentially small term that accounts for their singularity probability:

\begin{equation*}
\Pr\Big\{s_n(A)\le \frac{\eps}{\sqrt{n}}\Big\} \le \eps +2e^{-cn}, \quad \eps \ge 0.
\end{equation*}

After early breakthroughs of Rudelson \cite{Rudelson05} and Tao and Vu \cite{TV09}, it was shown in a remarkable work by Rudelson and Vershynin \cite{RV08} that Spielmann-Teng's conjecture holds (up to a multiplicative constant) for all random matrices with i.i.d. centered subgaussian entries of variance $1$. This result has been greatly extended and refined in subsequent works \cite{Livshyts18,LTV19,RT18,RV09,Tatarko,Vershynin11}. In particular, Livshyts, Tikhomirov and Vershynin \cite{LTV19} established the same bound as Rudelson and Vershynin for random matrices $A$ whose entries are independent random variables satisfying a uniform anti-concentration estimate, and such that the expected sum of the squares of the entries is $O(n^2)$.
 
\subsection{Random matrices with dependent entries}
The smallest singular value problem becomes significantly more difficult when one considers models of random matrices with dependencies between entries, and much less has been known \cite{ACW16,BCC12,CMMM19,FJLS20,Jain20,LLTTY17,LLTTY19,Meszaros18,Nguyen13,Nguyen14,NM18,NV13}. For a survey of the topic we refer the reader to \cite{Vu20}. Our main focus is on a simple model of random combinatorial matrices with non-independent entries. We begin with a brief discussion of relevant results.

For an even integer $n$, let $Q_n$ be a random $n\times n$ matrix  with entries in $\{0,1\}$ whose rows are independent vectors of exactly $n/2$ zero components. One can view $Q_n$ as the bipartite adjacency matrix of a random $n/2$-regular bipartite graph with parts of size $n$.
Define $p_n$ to be the probability that $Q_n$ is singular. Trivially,
\[
p_n \ge \frac{1}{\binom{n}{n/2}}=(\frac12 +o(1))^n
\]
as the RHS is the probability that the first two rows of $Q_n$ are equal. Nguyen \cite[Conjecture 1.4]{Nguyen13} conjectured that this simple bound is close to the truth.

\begin{conjecture}[Nguyen  \cite{Nguyen13}]\label{conj:Nguyen}
	$p_n=(\frac12 +o(1))^n$.
\end{conjecture}

It is already non-trivial to justify that $p_n$ approaches $0$ as $n$ tends to infinity. This was first achieved by Nguyen \cite{Nguyen13}, who proved that $p_n$ decays faster than any polynomial, using an inverse Littlewood-Offord theorem of Nguyen and Vu \cite{NV11}. Recently, Ferber, Jain, Luh and Samotij \cite{FJLS20} established a counting result for the inverse Littlewood-Offord problem to improve the bound further to $p_n \le \exp(-n^{0.1})$. 

The question of obtaining quantitative lower tail bounds on the least singular value of $Q_n$ was considered by Nguyen
and Vu \cite{NV13}, where it was shown that for any $C>0$ there exists $D>0$ for which $\Pr\Big\{s_n(Q_n) \le n^{-D}\Big\} \le n^{-C}$. Building on the work of Ferber et al., Jain \cite{Jain20} obtained a better bound:

\[
\Pr\Big\{s_n(Q_n)\le \frac{\eps}{n^2}\Big\} \le C\eps + 2e^{-n^{0.0001}}.
\]

The following is our main result.

\begin{thm}\label{thm:row regular}
There exist constants $C,c>0$ such that
\[
\Pr\Big\{s_n(Q_n)\le \frac{\eps}{\sqrt{n}}\Big\} \le C\eps + 2e^{-cn}, \quad \eps \ge 0. 
\]
\end{thm}

{\bf Remark.}  (1) \cref{thm:row regular} improves on the aforementioned works of Ferber, Jain, Luh and Samotij \cite{FJLS20}, and Jain \cite{Jain20}. 

\vspace*{0.2cm}
(2) For $\eps=0$, \cref{thm:row regular} yields the first exponential bound for the singularity probability:
\[
\Pr\big\{Q_n \enskip \text{is singular}\big\} \le 2 e^{-cn}.
\]

\vspace*{0.2cm}
(3) \cref{thm:row regular} continues to hold in the more general case when the sum of each row is $d$, where $\min(d,n-d)\gtrsim n$. Here we have focused on the case $n$ is even and $d=n/2$ for ease of exposition.

\subsection{The main tools}

The {\em L\'evy concentration function} of a random variable $\xi$ is defined for $\eps>0$ as
\[
\cL(\xi,\eps):=\sup_{x}\Pr\{|\xi-x|<\eps\}.
\] 

The following theorem is our main tool in proving \cref{thm:row regular}.

\begin{thm}[Distances] \label{thm:distances}
Let $R_1,\ldots,R_n$ denote the row vectors of $Q_n$, and consider the subspace $H_n=\spn(R_1,\ldots,R_{n-1})$.	Let $\bm{v}$ be a random unit vector orthogonal to $H_n$ and measurable with respect to the sigma-field generated by $H_n$. Then
\[
\cL\left(\product{R_n,\bm{v}}, \eps\right) \le C\eps + 2e^{-cn} \quad \text{for every } \eps \ge 0,
\]
where $C>0$ and $c\in (0,1)$ are constants. In particular, we have
\[
\Pr\big \{\dist(R_n,H_n) \le \eps \big\} \le C\eps + 2e^{-cn}.
\]
\end{thm}

A version of \cref{thm:distances}, under the assumption that the entries are i.i.d., was obtained by Rudelson and Vershynin \cite{RV08}. They quantified the amount of additive structure of a vector (in this case, a normal vector of $H_n$) by the Least Common Denominator (LCD). The authors of \cite{RV08} proved a {\em small ball probability} bound for weighted sums of i.i.d. random variables in terms of the LCD of the coefficient vector, and used it to estimate $\cL(\dist(R_n,H_n),\eps)$. However, in our model, the LCD is no longer applicable as the entries are not independent.

In the present paper, we develop a {\em combinatorial} version of the least common denominator and show how it can handle the dependent coordinates.

\begin{defn}[Combinatorial Least Common Denominator]
	For a vector $\bm{v} \in \bR^n$, and parameters $\al, \ga>0$, define
	\begin{equation*}
	\CLCD_{\al,\ga}(\bm{v}):=\inf\Big\{\theta>0\colon \dist(\theta\cdot \D(\bm{v}),\bZ^{\binom{n}{2}}) < \min\big(\ga \norm{\theta \cdot \D(\bm{v})},\al\big) \Big \}.
	\end{equation*}
	Here by $\D(\bm{v})$ we denote the vector in $\bR^{\binom{n}{2}}$ whose $(i,j)$-coordinate is $v_i-v_j$, for $1\le i<j \le n$.	
\end{defn}

{\bf Remark.} The requirement that the distance is smaller than 
$\ga \norm{\theta \cdot \D(\bm{v})}$ forces us to consider only non-trivial integer points as approximations of $\D(\bm{v})$. We will use this definition with $\ga$ a small constant, and for $\al=\mu n$ with a small constant $\mu>0$. The inequality $\dist(\theta\cdot \D(\bm{v}),\bZ^{\binom{n}{2}})<\al$ then yields that most coordinates of $\D(\bm{v})$ are within a small distance from non-zero integers.

One new key ingredient in the proof of \cref{thm:distances} is a small ball probability inequality for certain {\em combinatorial statistics}. Given a vector $\bm{v}=(v_1,\ldots,v_n)\in \bR^n$, consider the random sum
\[
W_{\bm{v}}:=\eta_1 v_1+\ldots+\eta_n v_n,
\]
where $(\eta_1,\ldots,\eta_n)$ is taken uniformly from $\{0,1\}^n$ subject to $\sum_{i=1}^{n}\eta_i=n/2$. We establish the following vital relation between CLCD of $\bm{v}$ and anti-concentration of $W_{\bm{v}}$, whose proof employs the analytic framework of Hal\'asz \cite{Halasz77} together with an estimate, due to Roos \cite{Roos19}, on the characteristic function of a combinatorial statistic. 

\begin{restatable}[Small ball probability]{thm}{binomial}	\label{binomial}	
	For any $b>0$ and $\ga\in (0,1)$ there exists $C>0$ depending only on $b$ and $\ga$ with the following property. Let $\bm{v}\in \bR^n$ such that $\norm{\D(\bm{v})}\ge b\sqrt{n}$. Then for every $\al>0$ and $\eps \ge 0$, we have
	\[
	\cL(W_{\bm{v}},\eps) \le C\eps + \frac{C}{\CLCD_{\al,\ga}(\bm{v})}+Ce^{-2\al^2/n}.
	\]	
\end{restatable}
Note that the requirement $\norm{\D(\bm{v})}\ge b\sqrt{n}$ holds for a \textquote{typical} vector $v$. 
Given this theorem, it is then natural to attack \cref{thm:row regular} using the geometric approach of Rudelson and Vershynin \cite{RV08}. However, for the proof of \cref{thm:row regular}, we have to introduce new novel techniques.

\subsection{Organization and notation} 
The paper is organized as follows. We prove \Cref{thm:row regular,thm:distances} in \cref{sec:smallest singular value}, assuming the validity of \cref{binomial}. \cref{sec:small ball probability} is devoted to the proof of \cref{binomial}. We close,
in \cref{sec:concluding remarks}, with some remarks and open problems.

The inner product in $\bR^n$ is denoted $\langle\cdot,\cdot\rangle$, the Euclidean norm is denoted $\norm{\cdot}$. The Euclidean unit ball and sphere in $\bR^n$ are denoted $B_2^n$ and $\bS^{n-1}$, respectively. We generally use lowercase bolded letters for vectors. For a vector $\bm{v}$ we denote by $v_i$ the value of its $i$-th coordinate. Given $S,P \subset \bR^n$, we use the standard notation $N(S,P)$ for the least number of translates of $P$ needed to cover $S$.

Suppose that $A=(a_{ij})_{1\le i \le m, 1 \le j \le n}$ is a $m\times n$ real matrix. The {\em operator norm} of $A$ is defined as $\lVert A\rVert_{\op}:=\sup_{\bm{v}\in \bS^{n-1}}\norm{A\bm{v}}$, while the {\em Hilbert-Schmidt norm} of $A$ is given by $\lVert A\rVert_{\HS}:=\big(\sum_{i,j}a_{ij}^2\big)^{1/2}$.

We write $[n]$ for the set $\{1,2,\ldots,n\}$, while $\binom{X}{k}$ is the family of all $k$-element subsets of a set $X$. 

Throughout the paper we omit floor and ceiling signs where the argument is unaffected. We will employ the following asymptotic notation: $f=O(g)$, $f \lesssim g$, $g\gtrsim f$ all mean that $|f| \le Cg$ for some absolute constant $C>0$, while $f \ll g$ means that $|f| \le cg$ for a sufficiently small constant $c>0$. We will indicate dependence of the implied constant on parameters with subscripts, e.g. $f\lesssim_{\al} g$. We will also use $C, c, c_1,c_2$, etc. to denote unspecified positive constants whose values may be different at each occurrence, and are understood to be absolute if no dependence on parameters is mentioned.

\section{Smallest singular value of random row-regular matrices}
\label{sec:smallest singular value}

This section is devoted to the proofs of \cref{thm:distances,thm:row regular}.
We will follow the general strategy of Rudelson and Vershynin \cite{RV08}, the first step of which is to decompose the unit sphere to a ``structured'' part $\calC(\de,\rho)$ and a ``generic'' one $\cN(\de,\rho)$. We carry out this step in \cref{sec:decomposition} (see \cref{defn:decomposition}). In \cref{sec:concentration} we show that the random variable $\norm{Q_nv}$ satisfies a large deviation inequality (\cref{concentration single vector}). We then use this inequality in \cref{sec:concentration} to achieve an essentially sharp upper bound for the restricted operator norm (see \cref{restricted norm}), and in \cref{sec:almost constant} to get a good uniform lower bound for $\norm{Q_n\bm{v}}$ on the set $\calC(\de,\rho)$ (see \cref{Invertibility constant vectors}). We investigate the invertibility problem for non almost-constant vectors in \cref{sec:non almost constant} using  a special case of our small ball probability estimate. Putting the pieces together, we deliver the proofs of \cref{thm:distances,thm:row regular} in \cref{sec:final hurdle}.

\subsection{Decomposition of the sphere}\label{sec:decomposition}

 We will make use of a partition of the unit sphere $\bS^{n-1}$ into two sets of almost-constant and non almost-constant vectors. These sets were first defined in \cite{LLTTY17, LLTTY19} as follows.

\begin{defn}\label{defn:decomposition}
Fix $\de,\rho\in(0,1)$ whose values will be chosen later. A vector $\bm{v} \in \bS^{n-1}$ is called {\em almost-constant} if one can find $\la \in \bR$ such that there are at least $(1-\de)n$ coordinates $i\in [n]$ satisfying $|v_i - \la|\le \frac{\rho}{\sqrt{n}}$. A vector $\bm{v} \in \bS^{n-1}$ is called {\em non almost-constant} if it is not almost-constant. The sets of almost-constant and non almost-constant vectors will be denoted by $\calC(\de,\rho)$ and $\cN(\de,\rho)$, respectively.
\end{defn}
Remark that $\calC(\de',\rho')\subseteq \calC(\de,\rho)$ if $\de' \le \de$ and $\rho' \le \rho$. 

Using the decomposition $\bS^{n-1}=\calC(\de,\rho)\cup \cN(\de,\rho)$ of the unit sphere, we divide the invertibility problem into two subproblems, for almost-constant and non almost-constant vectors:
\begin{equation*}
\Pr\Big\{s_{\min}(Q_n) \le \frac{\eps}{\sqrt{n}}\Big\} \le \Pr\Big\{\inf_{\bm{v}\in \calC(\de,\rho)} \norm{Q_n \bm{v}}\le \frac{\eps}{\sqrt{n}}\Big\}+\Pr\Big\{\inf_{\bm{v}\in \cN(\de,\rho)}\norm{Q_n\bm{v}}\le \frac{\eps}{\sqrt{n}}\Big \}.
\end{equation*}

We will deal with the former in \cref{sec:almost constant}, and the latter will be treated in \cref{sec:non almost constant}.

The fact that the operator norm of $Q_n$ typically has a higher order of magnitude compared to $\sqrt{n}$  adds some complexity to the proof. To overcome this difficulty, as in \cite{Jain20} we exploit the presence of a \textquote{spectral gap}. Namely we show in \cref{restricted norm} that, while the operator norm of $Q_n$ is $n/2$, the operator norm of $Q_n$ restricted to the hyperplane $\cH:=\{\bm{v}\in \bR^n\colon v_1+\ldots+v_n=0\}$ is $O(\sqrt{n})$ with high probability.

The following result is a version of \cite[Lemma 2.2]{LLTTY19}.

\begin{lem}[Non almost-constant vectors are separated]
	\label{non-constant separated}
	Let $\de,\rho \in(0,1)$. Then for any vector $\bm{v}\in \bS^{n-1} \setminus \calC(\de,\rho)$, there are disjoint subsets $\sigma_1=\sigma_1(\bm{v})$ and $\sigma_2=\sigma_2(\bm{v})$ of $[n]$ such that
	\begin{equation*}
	|\sigma_1|,|\sigma_2|\ge  \de n/8, \quad\mbox{ and } \quad \,\, \frac{\rho}{\sqrt{2n}} \le |v_i-v_j|\le \frac{6}{\sqrt{\de n}} \quad \forall i\in \sigma_1,\,\forall j\in \sigma_2.
	\end{equation*}
\end{lem}
\begin{proof}
	Consider the subset $\sigma\subseteq [n]$ defined as
	\[
	\sigma:=\Big\{k\colon |v_k| \le \frac{3}{\sqrt{\de n}}\Big\}.
	\]
	As $\norm{\bm{v}}=1$, we must have $|\sigma^{c}| \le \de n/9$.	
	
	Moreover, \cite[Lemma 2.2]{LLTTY19} guarantees the existence of disjoint subsets $J$ and $Q$ of $[n]$ such that
	\begin{equation*}
	|J|,|Q|\ge  \de n/4 \quad\mbox{and } \quad |v_i-v_j|\ge \frac{\rho}{\sqrt{2n}} \quad \forall i\in J,\,\forall j\in Q.
	\end{equation*}
	Put $\sigma_1=J\cap \sigma$ and $\sigma_2=Q\cap \sigma$. It is easy to verify that $\sigma_1$ and $\sigma_2$ possess the desired properties.
\end{proof}

\subsection{Concentration}\label{sec:concentration}

The main result of this section, stated below, shows that for each point $\bm{v}\in \bS^{n-1}$ the random variable $\norm{Q_n\bm{v}}^2$ is well-concentrated around its expectation.

\begin{lem}\label{concentration single vector}
	Let $M$ be a random $m\times n$ matrix, $1\le m\le n$, whose rows are independent random $\{0,1\}$-vectors of exactly $n/2$ zero components. Consider $\bm{v}\in \bS^{n-1}$, and let $r=\norm{v_1+\ldots+v_n}$. Then for every $t\ge 0$, one has
	\[
	\Pr\left\{\norm{\norm{M\bm{v}}^2-\bE\norm{M\bm{v}}^2}\ge t \right\} \le 2 \exp\left[-c_1\min\left(\frac{t^2}{(r^2+1)^2n},\frac{t}{r^2+1}\right)\right]
	\]
	where $c_1>0$ is an absolute constant.
\end{lem}

Our proof of \cref{concentration single vector} makes use of the following inequality, due to Kwan, Sudakov and the author \cite[Lemma 2.1]{KST19}.

\begin{lem}[Combinatorial concentration inequality] \label{concentration}
	Consider $f\colon \{0,1\}^n\rightarrow \bR$ such that
	\[
	\norm{f(x_1,\ldots,x_{i-1},1,x_{i+1},\ldots,x_n)-f(x_1,\ldots,x_{i-1},0,x_{i+1},\ldots,x_n)} \le d_i
	\]
	for all $\bm{x}\in \{0,1\}^n$ and all $i\in [n]$. Then
	\begin{align*}
		\Pr(|f(\bm{\eta})-\bE f(\bm{\eta})|\ge t) &\le 2\exp\left(-\frac{t^2}{8\sum_{i=1}^{n}d_i^2}\right) \quad \text{for all } t \ge 0,
	\end{align*}
where $\bm{\eta}$ is taken uniformly from $\{0,1\}^n$ subject to $\sum_{i=1}^{n}\eta_i=n/2$.	
\end{lem}

In order to justify \cref{concentration single vector}, we also require some results about random variables $X$ that satisfy $P(|X|>t)\le 2 \exp(-O(t^{\al}))$ for all $t\ge 0$, where $\al>0$ is fixed. The first is Proposition 2.5.2 in \cite{Vershynin18}.

\begin{lem}\label{lem:equivalent}
Fix $\al>0$. For a random variable $X$, the following properties are equivalent with parameters $K_i>0$ differing from each other by at most an absolute constant factor:
\begin{itemize}
\item[1.] Tails: $\Pr(|X| > t) \le 2\exp\left(-t^{\al}/K_1\right)$ for all $t\ge 0$;
\item[2.] Moments: $\left(\bE|X|^p\right)^{1/p} \le K_2p^{1/\al}$ for all $p\ge 1$.
\end{itemize}
\end{lem}

The above lemma leads us to the following convenient notation.

\begin{defn}[Orlicz norm]
Fix $\al>0$. The $\psi_{\al}$-norm of $X$, denoted $\norm{X}_{\psi_{\al}}$, is defined to be the smallest $K_2$ in the second property of \cref{lem:equivalent}. In other words,
\[
\norm{X}_{\psi_{\al}}:=\sup_{p\ge 1}p^{-1/\al}\left(\bE|X|^p\right)^{1/p}.
\]
\end{defn}

The cases $\al=2$ and $\al=1$ correspond to {\em sub-gaussian} random variables and {\em sub-exponential} random variables, respectively. Sub-gaussian and sub-exponential distributions are closely related. Indeed, inspecting the definitions we quickly see that
\begin{equation}\label{gaussian exponential}
\norm{X}_{\psi_2}^2 \le \norm{X^2}_{\psi_1} \le 2\norm{X}_{\psi_2}^2.
\end{equation}

We recall a concentration inequality for sums of independent sub-exponential random variables (see, e.g., \cite[Theorem 2.8.1]{Vershynin18}).

\begin{thm}[Bernstein's inequality] \label{Bernstein}
	Let $Y_1,\ldots,Y_n$ be independent, mean zero, sub-exponential random variables. Let $K=\max\limits_{i}\norm{Y_i}_{\psi_1}$. Then for every $t\ge 0$, we have
	\[
	\Pr\Big\{\Big|\sum_{i=1}^n Y_i \Big|\ge t\Big \} \le 2 \exp\left[-c_2\min\left(\frac{t^2}{K^2n},\frac{t}{K}\right)\right]
	\]
	where $c_2>0$ is an absolute constant.
\end{thm}

We are now ready to prove \cref{concentration single vector}.

\begin{proof}[Proof of \cref{concentration single vector}]
Consider the random sum $X:=\eta_1v_1+\ldots+\eta_nv_n$, where $(\eta_1,\ldots,\eta_n)$ is sampled uniformly from the set of all $\{0,1\}$-vectors of entry sum $n/2$. Then for all $i\in [m]$, the random variables $\product{M\bm{v},\bm{e}_i}$ are i.i.d. copies of $X$. Thus
\[
\norm{M\bm{v}}^2 =X_1^2+\ldots+X_m^2,
\]
where $X_1,\ldots,X_m$ are independent copies of $X$. 

A simple calculation shows
\[
\bE(X^2)=\frac{n-2}{4(n-1)}r^2+ \frac{n}{4(n-1)} \lesssim r^2+1.
\]
	
Moreover, applying \cref{concentration} to the function $\bm{\eta}\mapsto \norm{\product{\bm{\eta},\bm{v}}}$, we get

\begin{equation*}
\norm{X-\bE X}_{\psi_2}\lesssim 1.
\end{equation*}

By the triangle inequality for the $\psi_2$-norm, we obtain
	\begin{align*}
		\norm{X}_{\psi_2} &\le \norm{X-\bE X}_{\psi_2}+\norm{\bE X}_{\psi_2}\\
		&\le \norm{X-\bE X}_{\psi_2}+(\bE X^2)^{1/2} \lesssim 1 + (r^2+1)^{1/2}\lesssim r +1,
	\end{align*}
	where in the second inequality we used the bound $\norm{\bE X}_{\psi_2}=\norm{\bE X} \le (\bE X^2)^{1/2}$. Thus
	\begin{align*}
		\norm{X^2-\bE X^2}_{\psi_1} &\le \norm{X^2}_{\psi_1}+\norm{\bE X^2}_{\psi_1}\\
		&\le 2\norm{X}_{\psi_2}^2+\bE(X^2) \lesssim r^2+1,
	\end{align*}
where the second inequality follows from \eqref{gaussian exponential}. Now, noting that $m\le n$, and applying \cref{Bernstein} to $Y_i=X_i^2-\bE X_i^2$ and $K \lesssim r^2+1$, we obtain
	\[
	\Pr\left\{\norm{\norm{M\bm{v}}^2-\bE\norm{M\bm{v}}^2}\ge t \right\} \le 2 \exp\left[-c\min\left(\frac{t^2}{(r^2+1)^2n},\frac{t}{r^2+1}\right)\right] \quad \text{for all } t\ge 0,
	\]
	where $c>0$ is an absolute constant. This completes our proof.
\end{proof}

\subsection{Restricted operator norm and invertibility on a single vector}\label{sec:restricted}

Using nets along with our concentration inequality (\cref{concentration single vector}), one can show that the operator norm of $Q_n$ restricted to the hyperplane $\cH:=\{\bm{v}\in \bR^n:\sum_{i=1}^{n}v_i=0\}$ is typically $O(\sqrt{n})$.

\begin{prop}[Restricted operator norm]	\label{restricted norm}
There exist constants $C_3>0$ and $c_3>0$ such that the following holds. Let $M$ be a random $m\times n$ matrix, $1\le m\le n$, whose rows are independent random $\{0,1\}$-vectors of exactly $n/2$ zero components. Then for all $t\ge C_3$, we have
	\[
	\Pr\Big\{{\lVert\left. M \right|_{\cH}\rVert}_{\op} \ge t \sqrt{n}\Big \} \le 2 \exp\left(-c_3t^2n\right),
	\]
where $\cH:=\{\bm{v}\in \bR^n:\sum_{i=1}^{n}v_i=0\}$. 	
\end{prop}
\begin{proof}
	Note that
	\[
	{\lVert\left. M \right|_{\cH}\rVert}_{\op}=\sup_{\bm{v}\in \bS^{n-1}\cap \cH}\norm{M\bm{v}}.
	\]
	Let $\cN$ be a $(1/2)$-net of $\bS^{n-1}\cap \cH$ of cardinality at most $6^n$.
	Fix $\bm{v}\in \cN$. Since $v_1+\ldots +v_n=0$, it follows from \cref{concentration single vector} that for $t$ sufficiently large one has
	\[
	\Pr\left\{\norm{M\bm{v}}\ge t\sqrt{n}/2\right\} \le \Pr\left\{\norm{\norm{M\bm{v}}^2-\bE\norm{M\bm{v}}^2}\ge t^2n/8 \right\}\\
	\le 2 e^{-ct^2n},
	\]
	where the first inequality holds since $\bE\norm{M\bm{v}}^2=\frac{mn}{4(n-1)} \le n/2$. Taking the union bound yields
	\begin{align*}
		\Pr\left\{{\lVert\left. M \right|_{\cH}\rVert}_{\op} \ge t \sqrt{n}\right\} \le \norm{\cN}\max_{\bm{v}\in \cN}\Pr\left\{\norm{M\bm{v}}\ge t\sqrt{n}/2\right\} \le 6^n\cdot 2e^{-ct^2n},
	\end{align*}
	which completes the proof.
\end{proof}

\cref{concentration single vector} can also be used to establish the invertibility of the random matrix $Q_n$ on a single vector.

\begin{lem}[Invertibility on a single vector]\label{invertibility single vector}	
There exists a constant $c_4>0$ such that the following holds for fixed $\bm{v}\in \bS^{n-1}$. Let $M$ be a random $m\times n$ matrix, $n/2\le m\le n$, whose rows are independent random $\{0,1\}$-vectors of exactly $n/2$ zero components. Then
	\[
	\Pr\left\{\norm{M\bm{v}}\le \sqrt{n}/5\right\} \le 2e^{-c_4n}.
	\] 
\end{lem}
\begin{proof}
	 Let $r=\norm{v_1+\ldots+v_n}$.
	For $n \ge 10$ and for $m\ge n/2$, we have
	\begin{align*}
		\bE\norm{M\bm{v}}^2=\frac{m(n-2)}{4(n-1)}r^2+\frac{mn}{4(n-1)}\ge \frac19 (r^2+1)n.
	\end{align*}
	Applying \cref{concentration single vector} to $t=\frac{1}{18}(r^2+1)n$, we thus get
	\[
		\Pr\left\{\norm{M\bm{v}}\le \sqrt{n}/5\right\} \le \Pr\left\{\norm{\norm{M\bm{v}}^2-\bE\norm{M\bm{v}}^2}\ge t \right\}\le 2e^{-cn}. \qedhere
	\]
\end{proof}

{\bf Remark.} One can prove \cref{invertibility single vector} in a more direct way. Indeed, let us consider the random sum $X:=\eta_1v_1+\ldots+\eta_nv_n$, where $(\eta_1,\ldots,\eta_n)$ is taken uniformly at random from the set of all $\{0,1\}$-vectors of entry sum $n/2$. It is not difficult to derive from \cref{concentration} that $\Pr(\norm{X} \le c) \le 1-c$ for some constant $c>0$. The conclusion of \cref{invertibility single vector} then follows from a tensorization lemma (see \cref{tensorization}).

\subsection{Invertibility for almost-constant vectors}\label{sec:almost constant}

Here we study the invertibility problem for almost-constant vectors. The following is the main result.

\begin{prop}[Invertibility for almost-constant vectors] \label{Invertibility constant vectors}
There exist constants $\de,\rho,c_5 \in (0,1)$ such that the following holds. Let $M$ be an $m\times n$ random matrix, $n/2 \le m\le n$, whose rows are independent random $\{0,1\}$-vectors of exactly $n/2$ zero components. Then 
	\[
	\Pr\big\{ \inf_{\bm{v}\in \calC(\de,\rho)}\norm{M\bm{v}} \le \sqrt{n}/10\big\} \le 2 e^{-c_5n}.
	\]
\end{prop}

We will construct a small $\eps$-net $\cN$ on $\calC(\de,\rho)$ with respect to the pseudometric $d(\bm{x},\bm{y}):=\norm{M(\bm{x}-\bm{y})}$. The invertibility of the random matrix $M$ over a single vector $\bm{w}\in \cN$ will follows from \cref{invertibility single vector}. Then, by a union bound, the invertibility will hold for each point in the net $\cN$. By approximation, we will extend the invertibility to the whole $\calC(\de,\rho)$. 

To exploit the fact that ${\lVert\left.Q_n \right|_{\cH}\rVert}_{\op}=O(\sqrt{n})$, we will use the following simple result, whose proof borrows some ideas of Jain \cite{Jain20}.

\begin{lem}[Rounding]\label{rounding}
Fix $\be>0$, and consider any $S\subset \bS^{n-1}$. There exists a (deterministic) net	$\cN \subset S+2\be B_2^n$ of cardinality at most $(2n+2) \cdot N(S,\be B_2^n)$ such that for every $m\in \bN$ and for every (deterministic) $m\times n$ matrix $A$, the following holds: for every $\bm{v}\in S$ one can find $\bm{w}\in \cN$ so that
\[
\norm{A(\bm{v}-\bm{w})} \le \be \left( 2 {\lVert\left.A \right|_{\cH}\rVert}_{\op}+\frac{\lVert A\rVert_{\op}}{n}\right).
\]
\end{lem}

\begin{proof}
Let $\cF \subset \bR^n$ be a set such that $S\subset \cF+\be B_2^n$ and $\norm{\cF}=N(S,\be B_2^n)$. Consider the $(2n+2)$-element subset of $\frac{\be}{\sqrt{n}}\bZ^n$ defined as
\[
\cY:=\big\{\big(\be s/\sqrt{n},\ldots,\be s/\sqrt{n},0,\ldots,0\big)\colon s=\pm 1\big\}.
\]

We will show that the net $\cN:=\cF+\cY$ has the desired properties. Indeed, we have
\[
\norm{\cN}\le \norm{\cY}\norm{\cF}=(2n+2)\norm{\cF}.
\]
Fix a vector $\bm{v} \in S$. It follows from the definition of $\cF$ that $\norm{\bm{v}-\bm{x}} \le \be$ for some $\bm{x}\in \cF$.
Let $k:=\norm{\product{\bm{v}-\bm{x},\bm{1}}}\sqrt{n}/\be$, $s:=\sgn(\product{\bm{v}-\bm{x},\bm{1}})$, and

\begin{equation*}
\bm{y}:=(\underbrace{\be s/\sqrt{n},\ldots,\be s/\sqrt{n}}_\text{$\floor{k}$},0,\ldots,0) \in \cY.
\end{equation*}	

Note that $\bm{y}$ is well-defined since, by the Cauchy-Schwarz inequality, we have
\begin{align*}
k=\norm{\product{\bm{v}-\bm{x},\bm{1}}}\sqrt{n}/\be \le \norm{\bm{v}-\bm{x}}n/\be \le n. 
\end{align*}
Let $\bm{w}:=\bm{x}+\bm{y}$. Then $\bm{w}\in \cN$. Moreover,
by the triangle inequality we obtain
\[
\norm{\bm{v}-\bm{w}}\le \norm{\bm{v}-\bm{x}}+\norm{\bm{y}}\le \be +\be=2\be,
\]
and whence $\bm{w} \in \bm{v}+2\be B_2^n$. This implies 
\[
\cN \subset S+2\be B_2^n.
\]

We can infer from the definition of $\bm{y}$ that
\[
\product{\bm{v}-\bm{w},\bm{1}}=\product{\bm{v}-\bm{x},\bm{1}}-\product{\bm{y},\bm{1}}\\
=\frac{s\be k}{\sqrt{n}}-\frac{s\be \floor{k}}{\sqrt{n}}\in \Big[\frac{-\be}{\sqrt{n}},\frac{\be}{\sqrt{n}}\Big].
\]

Finally, writing $\bm{v}-\bm{w}=\Proj_{\cH}(\bm{v}-\bm{w})+\frac{\product{\bm{v}-\bm{w},\bm{1}}}{n}\bm{1}$, 
we see that
\begin{align*}
\norm{A(\bm{v}-\bm{w})} &\le \norm{A(\Proj_{\cH}(\bm{v}-\bm{w}))}+\frac{\norm{\product{\bm{v}-\bm{w},\bm{1}}}}{n}\norm{A \bm{1}}\\
&\le {\lVert\left.A \right|_{\cH}\rVert}_{\op}\cdot \norm{\bm{v}-\bm{w}}+\frac{\be}{n^{3/2}}\cdot \norm{A}_{\op}\cdot\sqrt{n}\\
&\le 2\be {\lVert\left.A \right|_{\cH}\rVert}_{\op}+\frac{\be \lVert A\rVert_{\op}}{n}.
\end{align*}

This completes our proof.
\end{proof}

We next employ \cref{rounding} to discretize the set of almost-constant vectors.

\begin{lem}[Discretization of almost-constant vectors] \label{discretization constant}
Let $\de,\rho \in (0,\frac{1}{12})$, and $n$ is sufficiently large with respect to $\de$ and $\rho$. Then there is a net $\cN \subset \bR^n$ of cardinality at most $e^{2\de \log (5/\de)n}$ such that 
\begin{itemize}
	\item For any $\bm{w}\in \cN$, we have $1/2 \le \norm{\bm{w}}\le 3/2$;
	\item For any $\bm{v}\in \calC(\de,\rho)$ there is $\bm{w}\in \cN$ so that for any deterministic $m\times n$ matrix $A$ we have
	\begin{equation}\label{eq:almost-constnat nets}
	\norm{A(\bm{v}-\bm{w})} \le (\de+2\rho)\left(2 {\lVert\left.A \right|_{\cH}\rVert}_{\op}+\frac{\lVert A\rVert_{\op}}{n}\right). 
	\end{equation}
\end{itemize}

\end{lem}
\begin{proof}
For a vector $\bm{v}\in \bR^n$ and a set $I\subseteq [n]$, we let $\bm{v}_I$ denote the vector $(v_i)_{i\in I}$. 

Fix $\bm{v}\in \calC(\de,\rho)$. Then there exist a real number $\la$ and an index set $\sigma\subset [n]$ with $\norm{\sigma}=(1-\de)n$ such that $\norm{v_i-\la} \le \frac{\rho}{\sqrt{n}}$ for all $i\in \sigma$. To discretize the range of $\la$, we observe that any $\la \in [-1,1]$ can be approximated by some 

\begin{equation}\label{eq:lambda}
\la_0 \in [-1,1]\cap \frac{\rho}{\sqrt{n}}\bZ
\end{equation}\label{eq:constant part}
 in the sense that $\norm{\la - \la_0} \le \frac{\rho}{\sqrt{n}}$. This clearly forces
\begin{equation*}
\norm{v_i-\la_0} \le \frac{2\rho}{\sqrt{n}} \quad \text{for all } i\in \sigma.
\end{equation*}

We can capture $\bm{v}_{\sigma^c}$ by quantizing its coordinates uniformly with step $\sqrt{\frac{\de}{n}}$. Thus there is $\tbu \in \sqrt{\frac{\de}{n}}\bZ^{\sigma^c}$ with $\lVert\bm{v}_{\sigma^c}-\tbu\rVert_{\infty} \le \sqrt{\frac{\de}{n}}$. Since $\norm{\bm{v}}=1$, we find
\begin{equation}\label{eq:tbu}
\norm{\tbu} \le \norm{\bm{v}_{\sigma^c}}+\norm{\bm{v}_{\sigma^c}-\tbu} \le \norm{\bm{v}}+\sqrt{\de n}\lVert\bm{v}_{\sigma^c}-\tbu\rVert_{\infty} \le 1+\de \le 3/2.
\end{equation}

Define a vector $\bm{u}\in \bR^n$ by setting $\bm{u}_{\sigma}=(\la_0,\ldots,\la_0)$ and $\bm{u}_{\sigma^c}=\tbu$. It follows from \eqref{eq:lambda} and \eqref{eq:tbu} that $\bm{u}\in \cF$, where
\begin{align}\label{initial net}
\cF:=\bigcup_{\norm{\sigma}=(1-\de)n}\Big\{\la_0\bm{1}_{\sigma}\colon \la_0\in [-1,1]\cap \frac{\rho}{\sqrt{n}}\bZ\Big\}\oplus \Big\{\tbu \in \sqrt{\frac{\de}{n}}\bZ^{\sigma^c}: \norm{\tbu} \le 3/2\Big\},
\end{align}
the union being over all $(1-\de)n$-element subsets $\sigma$ of $[n]$.

From the definition of $\tbu$ we have
\[
\norm{\bm{v}-\bm{u}} \le \sqrt{\de n}\lVert\bm{v}_{\sigma^c}-\tbu\rVert_{\infty}+\sqrt{n}\max_{i\in \sigma}\norm{v_i-\la_0} \le \de+2\rho.
\]
Letting $\be:=\de + 2\rho \in (0,\frac14)$, we thus get $\calC(\de,\rho)\subset \cF + \be B_2^n$, and so $N(\calC(\de,\rho),\be B_2^n) \le \norm{\cF}$. By applying \cref{rounding}, we obtain a net $\cN \subset \calC(\de,\rho)+2\be B_2^n \subset \frac32 B_2^n\setminus \frac12 B_2^n$ of cardinality  at most $(2n+2)\norm{\cF}$ having property \eqref{eq:almost-constnat nets}.

It remains to estimate the cardinality of $\cN$. Observe that there are $\binom{n}{\de n}\le (\frac{e}{\de})^{\de n}$ ways to choose the subset $\sigma$ in \eqref{initial net}. Clearly, there are at most $1+2\sqrt{n}/\rho$ possibilities for $\la_0$ in \eqref{initial net}. Furthermore, a known volumetric argument shows that there are at most $\left(\frac{5}{\de}\right)^{\de n}$ choices for $\tbu$ in \eqref{initial net}. Therefore, we have
\[
\norm{\cN} \le (2n+2)\norm{\cF}\le (2n+2)\cdot \left(\frac{e}{\de}\right)^{\de n}\cdot (1+2\sqrt{n}/\rho) \cdot \left(\frac{5}{\de}\right)^{\de n} \le e^{2\de \log (5/\de)n}
\]
for $n$ sufficiently large. This completes our proof.
\end{proof}

We are now ready to prove \cref{Invertibility constant vectors}.

\begin{proof}[Proof of \cref{Invertibility constant vectors}]
By \cref{restricted norm}, there is a constant $K\ge 1$ such that 
\[
\Pr\{{\lVert\left. M\right|_{\cH}\rVert}_{\op}>K\sqrt{n}\} \le e^{-n}.
\]
Take $\de=\rho=1/(100K)$. To complete the proof, it suffices to find a constant $c>0$ so that the event
\[
\cE:=\Big\{\inf_{\bm{v}\in \calC(\de,\rho)}\norm{M\bm{v}}\le \sqrt{n}/10 \enskip \text{and} \enskip {\lVert\left.M \right|_{\cH}\rVert}_{\op}\le K\sqrt{n}\Big \}
\]
has probability at most $2e^{-cn}$.

To this end, let $\cN$ be the net constructed in \cref{discretization constant}. By \cref{invertibility single vector}, for each $\bm{w}\in \cN$ we have
\[
\Pr\big \{\norm{M\bm{w}} \le \sqrt{n}/5\big \} \le 2e^{-c_4n}.
\]
Then taking the union bound, we obtain
\begin{equation}\label{eq:constant invertibility nets}
\Pr\big\{\inf_{\bm{w}\in \cN}\norm{M\bm{w}}\le \sqrt{n}/5\big \} \le e^{2\de \log (5/\de)n}\cdot  2e^{-c_4n} \le 2e^{-c_4n/2}.
\end{equation}

We are now in a position to bound the event $\cE$. Suppose that $\cE$ occurs, then ${\lVert\left.M \right|_{\cH}\rVert}_{\op}\le K\sqrt{n}$ and $\norm{M\bm{v}} \le \sqrt{n}/10$ for some $\bm{v}\in \calC(\de,\rho)$. We learn from the choice of $\cN$ that there is $\bm{w}\in \cN$ with 
\[
\norm{M(\bm{v}-\bm{w})} \le (\de+2\rho)\left( 2 {\lVert\left.M \right|_{\cH}\rVert}_{\op}+\frac{\lVert M\rVert_{\op}}{n}\right).
\]
Since ${\lVert\left.M \right|_{\cH}\rVert}_{\op}\le K\sqrt{n}$ and $\lVert M\rVert_{\op}\le n$, we have
\[
\norm{M\bm{w}} \le \norm{M\bm{v}}+\norm{M(\bm{v}-\bm{w})} \le \sqrt{n}/10+ (\de+2\rho)(2K\sqrt{n}+1) \le \sqrt{n}/5
\]
for $\de=\rho=1/(100K)$. By \eqref{eq:constant invertibility nets}, this completes our proof.	
\end{proof}

\subsection{Invertibility for non almost-constant vectors}
\label{sec:non almost constant}

In this section, we study the invertibility problem for non almost-constant vectors. The following is the main result.

\begin{prop}[Random normal] \label{random normal}
	There exist constants $\mu,\ga,c_6 \in (0,1)$ such that for $n$ sufficiently large one has	
	\[
	\Pr\big\{\exists \bm{v}\in \cN(\de,\rho) \enskip \text{with} \enskip Q_n'\bm{v}=0 \enskip \text{and} \enskip \CLCD_{\mu n,\ga}(\bm{v})\le e^{c_6 n}\big \}\le 2^{-n}.
	\]
\end{prop}

In \cref{sec:CLCD} we establish some properties of CLCD which are necessary for the proof of \cref{random normal}. The proof is then given in \cref{sec:Level sets}.  

\subsubsection{Properties of CLCD}\label{sec:CLCD}

A crucial property of the CLCD which will allow us to discretize the range of possible realizations of random normals, is {\em approximately stability of $\CLCD$ with respect to small perturbations}.

\begin{lem}[Stability of CLCD]\label{stable-CLCD}
	Consider a vector $\bm{v}\in \bR^n$, 
	and parameters $\al>0,\ga \in (0,1)$. 
	Then for any $\bm{w}\in\bR^n$ with $|\bm{v}-\bm{w}|<\frac{\ga \norm{\D(\bm{v})}}{5\sqrt{n}}$, we have
	\[
	\CLCD_{\al/2,\ga/2}(\bm{w}) \ge \min\Big\{\CLCD_{\al,\ga}(\bm{v}),\frac{\al}{4\sqrt{n}\norm{\bm{v}-\bm{w}}
	}\Big \}.
	\]
\end{lem}
\begin{proof} Note that $\norm{D(\bm{x})} \le \sqrt{n}\norm{x}$ for every $\bm{x}\in \bR^n$. By our assumptions on $\norm{\bm{v}-\bm{w}}$ and $\ga$, we get
\[
\norm{\D(\bm{v})-\D(\bm{w})}=\norm{\D(\bm{v}-\bm{w})} \le \sqrt{n}\norm{\bm{v}-\bm{w}}\le \sqrt{n}\cdot \frac{\ga \norm{\D(\bm{v})}}{5\sqrt{n}}<\norm{D(\bm{v})}/5,
\]
and hence
\[
\norm{\D(\bm{v})} \le 5\norm{\D(\bm{w})}/4.
\]
	Let $H:=\min\big\{\CLCD_{\al,\ga}(\bm{v}),\frac{\al}{4\sqrt{n}\norm{\bm{v}-\bm{w}}
	}\big \}$. For any $\theta \in (0,H)$, the definition of CLCD yields
	\[
	\dist\big(\theta\cdot \D(\bm{v}),\bZ^{\binom{n}{2}}\big)\ge \min\big(\ga \theta \norm{\D(\bm{v})}, \al\big).
	\]
	
From this it follows that	
	\begin{align*}
	\dist\big(\theta\cdot \D(\bm{w}),\bZ^{\binom{n}{2}}\big) &\ge \dist\big(\theta\cdot \D(\bm{v}),\bZ^{\binom{n}{2}}\big)-\norm{\theta \cdot \D(\bm{v}-\bm{w})}\\
	&\ge \min\big(\ga \theta \norm{\D(\bm{v})}, \al\big) - \theta \norm{\D(\bm{v}-\bm{w})}\\
	&\ge \min\big(\ga \theta \norm{\D(\bm{w})}, \al\big)-(1+\ga)\theta\norm{\D(\bm{v}-\bm{w})}\\
	&\ge \min\big(\ga \theta \norm{\D(\bm{w})},\al)-2\theta \sqrt{n}\norm{\bm{v}-\bm{w}}\ge \tfrac12 \min\big(\ga \theta \norm{\D(\bm{w})}, \al\big), 
	\end{align*}
where the last step holds since $\theta< \frac{\al}{4\sqrt{n}\norm{\bm{v}-\bm{w}}}$ and $4\sqrt{n}\norm{\bm{v}-\bm{w}} \le \tfrac45 \ga \norm{\D(\bm{v})} \le \ga \norm{\D(\bm{w})}$. By definition of CLCD, this gives
	\[
	\CLCD_{\al/2,\ga/2}(\bm{w}) \ge \theta. 
	\]
Since $\theta \in (0,H)$ was arbitrary, it follows that $\CLCD_{\al/2,\ga/2}(\bm{w}) \ge H$, which proves the lemma.
\end{proof}

We will also need a simple result that the CLCD of any non almost-constant vector in $\bS^{n-1}$ is $\gtrsim \sqrt{n}$.

\begin{lem}[Non almost-constant vectors have large CLCD] \label{clcd non-constant}
	Let $\de,\rho \in (0,1)$, and fix $\bm{v}\in \cN(\de,\rho)$. Then for every $\al>0$ and every $\ga$ with $0<\ga<\frac{1}{12}\de\rho$, we have
	\[
	\CLCD_{\al,\ga}(\bm{v}) \ge \tfrac{1}{7}\sqrt{\de n}.
	\]
\end{lem}
\begin{proof}
By \cref{non-constant separated}, there is a subset $\sigma'\subseteq \binom{[n]}{2}$ of cardinality
	\begin{equation*}
	|\sigma'| \ge \tfrac{1}{64}\de^2 n^2
	\end{equation*}
	and such that
	\begin{equation}\label{lower}
	\frac{\rho}{\sqrt{2n}} \le |v_i-v_j|\le \frac{6}{\sqrt{\de n}} \quad \text{for every} \enskip \{i,j\}\in \sigma'. 
	\end{equation}
	
	Let $H=\CLCD_{\al,\ga}(\bm{v})$. By definition of CLCD, one can find a vector $\bm{p}=(p_{ij})_{i<j}\in \bZ^{\binom{n}{2}}$ with
	\[
	\norm{H\cdot \D(\bm{v})-\bm{p}} < \ga H\norm{\D(\bm{v})}.
	\]
	Dividing by $H$ yields
	\[
	\norm{\D(\bm{v})-\frac{\bm{p}}{H}} < \ga \norm{\D(\bm{v})} \le \ga \sqrt{n}.
	\]
	Then by Chebyshev inequality, there exists a subset $\sigma''\subseteq \binom{[n]}{2}$ of cardinality
	\begin{align*}
	|\sigma''| >\binom{n}{2}-\tfrac{1}{64}\de^2n^2
	\end{align*}
	and such that
	\begin{equation}\label{upper}
	\norm{v_i-v_j-\frac{p_{ij}}{H}}<\frac{8\ga}{\de \sqrt{n}} \quad \text{for } \{i,j\}\in \sigma''.
	\end{equation}
	As $|\sigma'|+|\sigma''|>\binom{n}{2}$, there is $\{i,j\} \in \sigma'\cap \sigma''$. Fix this pair $\{i,j\}$. It follows from \eqref{lower}, \eqref{upper} and our assumption on $\ga$ that
	\[
	\norm{\frac{p_{ij}}{H}} \ge \norm{v_i-v_j}-\norm{v_i-v_j-\frac{p_{ij}}{H}}\ge  \frac{\rho}{\sqrt{2n}}-\frac{8\ga}{\de \sqrt{n}}>0,
	\]
	which shows $p_{ij}\ne 0$. Similarly, we have
	\[
	\norm{\frac{p_{ij}}{H}} \le \norm{v_i-v_j}+\norm{v_i-v_j-\frac{p_{ij}}{H}}\le \frac{6}{\sqrt{\de n}}+\frac{8\ga}{\de \sqrt{n}}<\frac{7}{\sqrt{\de n}}.
	\]
	This implies $H \ge \frac17 \norm{p_{ij}}\sqrt{\de n} \ge \frac17 \sqrt{\de n}$, completing our proof.
\end{proof}

\subsubsection{Level sets}\label{sec:Level sets}

In this section, we partition $\bS^{n-1}\setminus \calC(\de,\rho)$ into {\em level sets} collecting unit vectors having comparable CLCD. To show that with a high probability the normal vector does not belong to a level set with a small CLCD, we construct an approximating set whose cardinality is well controlled from above. Since CLCD is stable with respect to small perturbations, the event that the normal vector has a small CLCD is contained in the event that one of the vectors in the approximating set has a small CLCD. We then apply the small ball probability estimate for individual vectors, combined with the union bound, to show that the latter event has probability close to zero.

Unless stated otherwise, we will assume throughout this section that $\de,\rho, \mu$ and $\ga$ are constants with
\begin{equation}\label{eq:assumption}
0< \de, \rho \ll 1, \quad 0<\mu \ll_{\de,\rho} \ga \ll_{\de,\rho} 1.
\end{equation}
Let $H_0:=\frac17\sqrt{\de n}$. By \cref{clcd non-constant},
\[
\CLCD_{\al,\ga}(\bm{v}) \ge H_0 \quad \text{for every} \enskip \bm{v}\in \cN(\de,\rho).
\]

\begin{defn}[Level sets of CLCD] \label{defn: level sets}
Let $H\ge H_0/2$. We define the level set $S_H \subseteq \bS^{n-1}$ as
\[
S_H:=\{\bm{v}\in \cN(\de,\rho) \colon H\le \CLCD_{\mu n,\ga}(\bm{v})\le 2H\}.
\]
\end{defn}

We recall the following \textquote{tensorization} lemma
of Rudelson and Vershynin \cite[Lemma 2.2]{RV08}.

\begin{lem}[Tensorization lemma]\label{tensorization}
	Suppose that $\eps_0\in(0,1)$, $B\ge 1$, and let $X_1,\dots,X_m$ be independent random variables such that each $X_i$ satisfies 
	\[
	\Pr\{|X_i|\le \eps\}\leq B\eps\quad\mbox{for all }\eps\ge\eps_0.
	\]
	Then
	\[
	\Pr\left\{\norm{(X_1,X_2,\ldots,X_m)} \le \eps \sqrt{m} \right\}\le (CB\eps)^m \quad \text{for every } \eps\ge\eps_0,
	\]
	where $C>0$ is an universal constant.
\end{lem}

One can use the tensorization lemma to control the anti-concentration of $\norm{Q_n\bm{v}}$ where $\bm{v}$ is a fixed vector. Indeed, let $R_1,\ldots, R_n$ denote the (independent) rows of $Q_n$. Then $\norm{Q_n\bm{v}}^2=\sum_{i=1}^n\product{R_i,\bm{v}}^2
$, and we can apply \cref{tensorization} to $X_i:=\product{R_i,\bm{v}}$. Moreover, we can use \cref{binomial} to bound the L\'evy concentration function of each $X_i$. This gives:

\begin{lem}[Invertibility on a single vector via small ball probability]\label{invertibility single vector via small ball}
	For any $b>0$ and $\mu,\ga \in (0,1)$ there exist $c_7=c_7(b,\ga,\mu)>0$ and $C_7=C_7(b,\ga)>0$ such that the following holds. For any $\bm{v}\in\bR^n$ with $\norm{\D(\bm{v})}\ge b\sqrt{n}$ and any $\eps \ge \frac{1}{\CLCD_{\mu n,\ga}(\bm{v})}+e^{-c_7n}$, we have
	\[
	\Pr\big\{\norm{Q_n\bm{v}}\le \eps\sqrt{n}\big\}
	\le (C_7\eps)^n.
	\]
\end{lem}

To run the covering argument, we need the following discretization of the level set $S_H$.
 
\begin{lem}[Discretization of level sets]\label{discretization level sets}
Assume that the parameters $\de,\rho,\mu$ and $\ga$ satisfy \eqref{eq:assumption}.
Then there exists a net $\cN \subset S_H+\frac{8\mu\sqrt{n}}{H} B_2^n$ of cardinality at most $\mu^{-2}H^2\cdot (C_8H/\sqrt{n})^n$ with the following properties.
\begin{itemize}
	\item[\rm (P1)] For every $\bm{w} \in \cN$, one has $\CLCD_{\al/2,\ga/2}(\bm{w}) \ge H/32$ and $\D(\bm{w})\gtrsim_{\de,\rho} \sqrt{n}$. 
	\item[\rm (P2)] For every (deterministic) $m\times n$ matrix $A$, and for any $\bm{v}\in S_H$, one can find $\bm{w}\in \cN$ so that 
	\[
	\norm{A(\bm{v}-\bm{w})}\le \frac{4\mu\sqrt{n}}{H}\left( 2 {\lVert\left.A \right|_{\cH}\rVert}_{\op}+\frac{\lVert A\rVert_{\op}}{n}\right).
	\]
\end{itemize}
\end{lem}

{\bf Remark.} The property (P1) is in fact redundant since every vector $\bm{w}$ in $S_H+\frac{8\mu\sqrt{n}}{H} B_2^n$ satisfies (P1). However, it allows one to simplify the presentation. To prove the claim, let us consider any $\bm{w}\in S_H+\frac{8\mu\sqrt{n}}{H} B_2^n$. Take $\bm{v}\in S_H$ so that $\norm{\bm{v}-\bm{w}} \le 8\mu\sqrt{n}/H \lesssim_{\de} \mu$. Since $\bm{v}\in S_H \subseteq \cN(\de,\rho)$, \cref{non-constant separated} shows $\norm{\D(\bm{v})} \gtrsim_{\de,\rho}\sqrt{n}$. Therefore, $\norm{\bm{v}-\bm{w}}<\frac{\ga\norm{\D(\bm{v})}}{5\sqrt{n}}$ for $\mu \ll_{\de,\rho} \ga$, and so \cref{stable-CLCD} is applicable. We then get 
\[
\CLCD_{\mu n/2,\ga/2}(\bm{w}) \ge \min\Big\{\CLCD_{\al,\ga}(\bm{v}), \frac{\mu \sqrt{n}}{4\norm{\bm{v}-\bm{w}}}\Big \} \ge H/32
\]
as $\CLCD_{\al,\ga}(\bm{v})\ge H$ and $\norm{\bm{v}-\bm{w}}\le 8\mu\sqrt{n}/H$.
Moreover, the triangle inequality gives
\[
\norm{\D(\bm{w})} \ge \norm{\D(\bm{v})}-\norm{\D(\bm{v}-\bm{w})} \ge \norm{\D(\bm{v})}-\sqrt{n}\norm{\bm{v}-\bm{w}}\ge \norm{\D(\bm{v})}/2 \gtrsim_{\de,\rho} \sqrt{n},
\]
where in the third inequality we used the bound $\norm{\bm{v}-\bm{w}}\le \frac{\ga \D(\bm{v})}{5\sqrt{n}} \le \frac{\D(\bm{v})}{5\sqrt{n}}$. This completes the proof of our claim.

It remains to construct a net $\cN \subset S_H+\frac{8\mu\sqrt{n}}{H} B_2^n$ of cardinality at most $\mu^{-2}H^2\cdot (C_8H/\sqrt{n})^n$ satisfying (P2), a task we now begin.

\begin{proof}[Proof of \cref{discretization level sets}]
We will explore the additive structure of $S_H$ to construct a small net. To this end, fix $\bm{v}=(v_1,\ldots,v_n)\in S_H$, and let
	\[
	\sigma:=\Big\{i\in [n]: \norm{v_i}\le \sqrt{2/n}\Big\}.
	\]
	Since $\norm{\bm{v}}=1$, it follows from Chebyshev inequality that
	\begin{equation}\label{eq:approx-I}
	\norm{\sigma} \ge n/2.
	\end{equation}
	
	Denote $T:=\CLCD_{\mu n,\ga}(\bm{v})$.
	By the definition of $S_H$, we have $H\le T<2H$. 
	
	According to the definition of CLCD, there exists an integer vector $\bm{p}=(p_{ij})_{1\le i<j\le n}$	such that
	\begin{equation}\label{eq:approx-II}
	\norm{T\cdot \D(\bm{v})-\bm{p}}<\mu n.
	\end{equation}
	
	For $1\le j \le n$, consider the vectors $\bm{v}^{(j)}\in \bR^{n-1}$ and $\bm{p}^{(j)}\in \bZ^{n-1}$ defined as follows
	\[
	\bm{v}^{(j)}:=(v_1-v_j,\ldots,v_{j-1}-v_j,v_j-v_{j+1},\ldots,v_j-v_n), \quad \bm{p}^{(j)}:=(p_{1j},\ldots,p_{j-1 j},p_{j j+1},\ldots,p_{jn}).
	\]
	
	It follows from \eqref{eq:approx-II} that 
	\[
	\sum_{j\in [n]}|T\bm{v}^{(j)}-\bm{p}^{(j)}|^2=2\norm{T\cdot \D(\bm{v})-\bm{p}}^2 < 2(\mu n)^2.
	\]
	
	Noting that $|\sigma|\ge n/2$ by \eqref{eq:approx-I}, and using the pigeonhole principle, we thus get an index $j\in \sigma$ with 
	\[
	|T\bm{v}^{(j)}-\bm{p}^{(j)}|^2 \le \frac{2(\mu n)^2}{|\sigma|}\le 4\mu^2n.
	\]
	
	Taking the square root of both sides, and dividing by $T$ gives
	\begin{align}\label{approximate star}
		|\bm{v}^{(j)}-\frac{\bm{p}^{(j)}}{T}| \le \frac{2\mu \sqrt{n}}{T} \le \frac{2\mu\sqrt{n}}{H}\quad \text{for some } j \in [n].
	\end{align}
	
	By the inequality $(x+y)^2\le 2x^2+2y^2$, we get
	\begin{align*}
	|\bm{v}^{(j)}|^2 &\le 2n v_j^2+2(v_1^2+\ldots + v_n^2) \le 6,
	\end{align*}
	 where in the last step we used the bound $\norm{v_j}\le \sqrt{2/n}$ along with our assumption that $\norm{\bm{v}}=1$.
	
	Combining this bound with \eqref{approximate star} yields
	\begin{equation}\label{p}
	|\bm{p}^{(j)}| \le T|\bm{v}^{(j)}|+|T \bm{v}^{(j)}-\bm{p}^{(j)}|\le 2H\cdot \sqrt{6}+2\mu\sqrt{n} \le 7H.
	\end{equation}
	
	To locate $\bm{v}$, we discretize the ranges of $v_j$ and $T$. Consider the lattice intervals
	\begin{equation*}
	\Lambda:=\frac{\mu}{2H}\bZ \cap [-1,1], \qquad \Theta= \frac{1}{7}\mu\bZ\cap [H,2H].
	\end{equation*}
Then one can find $\la_0 \in \Lambda$ and $T_0\in \Theta$ such that 
\[
\norm{v_j-\la_0} \le \frac{\mu}{2H}, \quad \norm{T-T_0} \le \mu/7.
\]
Letting $\bm{w}=(\frac{p_{1j}}{T_0}+\la_0,\ldots,\frac{p_{j-1 j}}{T_0}+\la_0,\la_0,-\frac{p_{j j+1}}{T_0}+\la_0,\ldots,-\frac{p_{jn}}{T_0}+\la_0)$, we see that

\begin{align}\label{w}
\notag |\bm{v}-\bm{w}|^2 &=\sum_{i=1}^{j-1}\left(v_i-\frac{p_{ij}}{T_0}-\la_0\right)^2+\sum_{k=j+1}^{n}\left(v_k+\frac{p_{jk}}{T_0}-\la_0\right)^2\\ \notag
& \le 3\sum_{i=1}^{j-1}\Big\{\left(v_i-v_j-\frac{p_{ij}}{T}\right)^2+(v_j-\la_0)^2+\left(\frac{p_{ij}}{T}-\frac{p_{ij}}{T_0}\right)^2 \Big\}+(v_j-\la_0)^2\\ \notag
& \qquad \qquad +3\sum_{k=j+1}^{n}\Big\{\left(v_k-v_j+\frac{p_{jk}}{T}\right)^2+(v_j-\la_0)^2+\left(\frac{p_{jk}}{T_0}-\frac{p_{jk}}{T}\right)^2 \Big\}\\
&=3|\bm{v}^{(j)}-\frac{\bm{p}^{(j)}}{T}|^2+(3n-2)(v_j-\la_0)^2+\left(\frac{1}{T}-\frac{1}{T_0}\right)^2|\bm{p}^{(j)}|^2\le \frac{14\mu^2 n}{H^2}
\end{align}	 
where the second line uses the inequality $(x+y+z)^2 \le 3(x^2+y^2+z^2)$, while the last step uses $|\bm{v}^{(j)}-\frac{\bm{p}^{(j)}}{T}|\le 2\mu\sqrt{n}/H$, $|v_j-\la_0| \le \mu/(2H)$, $|\frac{1}{T}-\frac{1}{T_0}|\le \mu/(7H^2)$ and $|\bm{p}^{(j)}| \le 7H$.

It follows from \eqref{p} and \eqref{w} that $\bm{v}$ is within Euclidean distance $4\mu \sqrt{n}/H$ from the set
\[
\cF_j:=\Big\{\Big(\frac{q_1}{T_0}+\la_0,\ldots,\frac{q_{j-1}}{T_0}+\la_0,\la_0,\frac{q_{j+1}}{T_0}+\la_0,\ldots,\frac{q_n}{T_0}+\la_0\Big)\colon \la_0\in\Lambda, T_0\in \Theta, \bm{q}\in \bZ^{[n]\setminus\{j\}}\cap B(0,7H)\Big\}.
\]

There are at most $1+4H/\mu$ choices for $\la_0\in\Lambda$, at most $1+7H/\mu$ ways to choose $T_0\in \Theta$, and at most $(1+21H/\sqrt{n})^n$ possibilities for the integer points $\bm{q}$ in $B(0,7H)$. This results in
\[
|\cF_j| \le (1+4H/\mu)\cdot (1+7H/\mu)\cdot (1+21H/\sqrt{n})^n \le \mu^{-2}H^2(CH/\sqrt{n})^n,
\]
where in the last inequality we used the assumption that $H \gtrsim_{\de}\sqrt{n}$.

Applying \cref{rounding} to $S=S_H$ and $\be=4\mu \sqrt{n}/H$, we therefore obtain a net $\cN \subset S_H+\frac{8\mu\sqrt{n}}{H}B_2^n$ of cardinality at most 
\[
(2n+2)\cdot\norm{\cF_1\cup \ldots \cup \cF_n}\le (2n+2)\cdot n\cdot \mu^{-2}H^2(CH/\sqrt{n})^n \le \mu^{-2}H^2\cdot (C_8H/\sqrt{n})^n
\]
with the desired properties. This completes our proof.
\end{proof}

\begin{lem}[Invertibility on a level set] \label{invertibility level set}
There exist constants $\mu,\ga,c_9 \in (0,1)$ and $C_9>0$ such that the following holds. Suppose that $n\ge C_9$ and $H_0 \le H \le e^{c_9n}$. Then
	\[
	\Pr\Big\{\inf_{\bm{v}\in S_H}\norm{Q'_n\bm{v}} \le c_9 n/H\Big \} \le 2e^{-n}.
	\]
\end{lem}
\begin{proof}
By \cref{restricted norm}, there exists a constant $K\ge 1$ such that 
\[
\Pr\{{\lVert\left.Q'_n \right|_{\cH}\rVert}_{\op}>K\sqrt{n}\} \le e^{-n}.
\]
	
Thus, in order to complete the proof, it suffices to find constants $C, c>0$ so that for $n\ge C$ and $H_0 \le H \le e^{cn}$, the event
\[
\cE:=\Big\{\inf_{\bm{v}\in S_H}\norm{Q'_n\bm{v}}\le \frac{c n}{2H}\enskip \text{and} \enskip {\lVert\left.Q'_n \right|_{\cH}\rVert}_{\op}\le K\sqrt{n}\Big \}
\]
has probability at most $e^{-n}$.
	
We claim that this holds with the following choice of parameters:
\[
c=\min\Big\{1/(8C_7C_8),c_7, 1\Big\}, \quad C=\max\{(64/c)^2, (C_7c)^{-2}\}, \quad 0<\de,\ga \ll 1 \enskip \text{and}  \enskip 0<\mu\ll_{\de,\rho, K} \ga,
\]	
where $C_7, c_7>0$ are the constants in  \cref{invertibility single vector via small ball}, and $C_8>0$ is the constant in \cref{discretization level sets}.	
Let $\cN$ be the net defined in \cref{discretization level sets}. Fix $\bm{w}\in \cN$. Then we have $\CLCD_{\mu n/2,\ga/2}(\bm{w}) \ge H/32$ and $\D(\bm{w})\gtrsim_{\de,\rho} \sqrt{n}$. We see that $\eps:=c\sqrt{n}/H$ satisfies $\eps \ge \frac{1}{\CLCD_{\mu n/2,\ga/2}(\bm{w})}+e^{-c_7n}$ assuming that $n\ge (64/c)^2$ and $c\le c_7$. Thus \cref{invertibility single vector via small ball} applies. We then get
\[
\Pr\Big \{\norm{Q_n'\bm{w}}\le \frac{c n}{H}\Big \} \le \left(\frac{C_7c \sqrt{n}}{H}\right)^{n-1}.
\]
By the union bound, noting that $n\ge (C_7c)^{-2}$ and $H \le e^{n}$, we have 
\begin{equation}\label{eq:invertibility on nets level set}
\Pr\Big \{\inf_{\bm{w}\in \cN}\norm{Q_n'\bm{w}}\le \frac{c n}{H}\Big \} \le \mu^{-2}H^2\left(\frac{C_8H}{\sqrt{n}}\right)^n \left(\frac{C_7c \sqrt{n}}{H}\right)^{n-1}\le \mu^{-2}H^3 \cdot (C_8C_7c)^n \cdot H \le e^{n}\cdot 8^{-n} \le  e^{-n}.
\end{equation}

Assume that the event $\cE$ holds.	Fix $\bm{v}\in S_H$ with $\norm{Q_n'\bm{v}}\le \frac{c n}{2H}$. From the definition of $\cN$, we see that for every fixed realization of $Q'_n$ there exists $\bm{w}\in \cN$ for which
\[
\norm{Q'_n(\bm{v}-\bm{w})}\le \frac{4\mu\sqrt{n}}{H} \left( 2{\lVert\left.Q'_n \right|_{\cH}\rVert}_{\op}+\frac{\lVert Q'_n\rVert_{\op}}{n}\right).
\]
By the triangle inequality we thus get
\[
\norm{Q_n'\bm{w}} \le \norm{Q_n'\bm{v}}+\frac{4\mu\sqrt{n}}{H}\left( 2 {\lVert\left.Q_n' \right|_{\cH}\rVert}_{\op}+\frac{\lVert Q_n'\rVert_{\op}}{n}\right) \le \frac{c n}{2H}+\frac{4\mu\sqrt{n}}{H}\cdot (2 K\sqrt{n}+1)\le \frac{c n}{H}
\]
as $\mu \ll_{K} 1$.

We have shown that the event $\cE$ implies the event that $\inf\limits_{\bm{w}\in \cN}\norm{Q_n'\bm{w}}\le \frac{c n}{H}$, whose probability is at most $e^{-n}$ due to \eqref{eq:invertibility on nets level set}. This completes our proof.
\end{proof}

Now we derive \cref{random normal} from \cref{invertibility level set}.

\begin{proof}[Proof of \cref{random normal}]
	Let $\mu,\ga, c_9 \in (0,1)$ be constants from \cref{invertibility level set}.	
	Consider the event
	\[
	\cE:=\big\{\exists \bm{v}\in \cN(\de,\rho) \enskip \text{with} \enskip Q_n'\bm{v}=0 \enskip \text{and} \enskip \CLCD_{\mu n,\ga}(\bm{v})\le e^{c_9n}\big \}.
	\]
	By \cref{clcd non-constant}, we have $\CLCD_{\mu n,\ga}(\bm{v})\ge H_0$ for all $\bm{v}\in \cN(\de,\rho)$, and so
	
	\[
	\Pr(\cE) \le \sum_{H_0/2\le 2^k\le e^{c_9 n}}\Pr\big\{\exists \bm{v}\in S_{2^k} \enskip \text{with} \enskip Q_n'\bm{v}=0\big \}.
	\]
	To estimate the sum, we apply \cref{invertibility level set}. We then get
	\[
	\Pr\big\{\exists \bm{v}\in S_{2^k} \enskip \text{with} \enskip Q_n'\bm{v}=0\big \} \le \Pr\big\{\inf_{\bm{v}\in S_{2^k}}\norm{Q_n'\bm{v}}\le c_9 n/2^k \big \} \le 2 e^{-n}
	\]
	for $n$ sufficiently large. Taking the union bound yields
	\[
	\Pr(\cE) \le c_9 n \log_2 e \cdot 2 e^{-n} \le 2^{-n}. \qedhere
	\]
\end{proof}

\subsection{Proofs of \Cref{thm:row regular,thm:distances}}\label{sec:final hurdle}

We first deduce \cref{thm:distances} from our small ball probability estimate and our bound on CLCD of the random normal.

\begin{proof}[Proof of \cref{thm:distances}]
Choose parameters $\de$ and $\rho$ such that $0<\de,\rho \ll 1$.
It follows from \cref{Invertibility constant vectors} that with probability at least $1-2e^{-c_5n}$ any unit vector orthogonal to $H_n$ is in $\cN
(\de,\rho)$. Indeed, we learn from \cref{Invertibility constant vectors} that
\[
\Pr\big\{\exists \bm{v}\in \calC(\de,\rho) \enskip \text{orthogonal to } H_n\big \} \le \Pr\big\{\inf_{\bm{v}\in \calC(\de,\rho)}\norm{Q_n'\bm{v}} \le \sqrt{n}/10\big \} \le 2e^{-c_5n}.
\]

Applying \cref{random normal} together with the above observation, we get
\[
\bm{v} \enskip \text{is in } \cN(\de,\rho) \enskip \text{and } \CLCD_{\mu n,\ga}(\bm{v}) \ge e^{c_6n} 
\]
with probability at least $1-2e^{-c_5n}-2^{-n}$. Application of \cref{binomial} finishes the proof.
\end{proof}

In the rest of this section we will prove \cref{thm:row regular}. 
For this purpose, fix some parameters $\de,\rho\in (0,1)$ whose values will be chosen later, and define the sets of sparse, compressible, and incompressible vectors as follows:
\begin{gather*}
\Sparse(\de):=\{\bm{x}\in \bS^{n-1}\colon \norm{\supp(\bm{x})} \le \de n\},\\
\Comp(\de,\rho):=\{\bm{x}\in \bS^{n-1}\colon \dist(\bm{x},\Sparse(\de)) \le \rho\},\\
\Incomp(\de,\rho):=\bS^{n-1}\setminus \Comp(\de,\rho).
\end{gather*}

Next we derive \cref{thm:row regular} from \cref{thm:distances}, using the \textquote{invertibility via distance} lemma from \cite{RV08}.

\begin{lem}[Invertibility via distance]\label{invertibility via distance}
Let $M$ be any random matrix. Let $R_1,\ldots, R_n$ denote the row vectors of $M$, and let $H_k$ denote the span of all row vectors except the $k$-th.	Then for every $\de,\rho \in (0,1)$ and every $\eps \ge 0$, one has
\[
\Pr\Big\{\inf_{\bm{x}\in \Incomp(\de,\rho)}\norm{\bm{x}^{\intercal}Q_n}\le \eps \frac{\rho}{\sqrt{n}}\Big\} \le \frac{1}{\de n}\sum_{k=1}^{n}\Pr\big\{\dist(R_k,H_k)\le \eps\big\}.
\]
\end{lem}

The proof of \cref{thm:row regular} also makes use of the following result, which gives a good uniform lower bound for $\bm{x}^{\intercal}Q_n$ on the set of compressible vectors.

\begin{prop}\label{anticoncentration estimate}
There exist constants $\de,\rho, c_{10}\in (0,1)$ such that
\[
\Pr\Big\{\inf_{\bm{x}\in \Comp(\de,\rho)} \norm{\bm{x}^{\intercal}Q_n}\le \sqrt{n}/270\Big\} \le 2e^{-c_{10}n}.
\] 
\end{prop}

Before proceeding with the proof of \cref{anticoncentration estimate}, we show how to deduce \cref{thm:row regular} from \cref{invertibility via distance} and \cref{anticoncentration estimate}.

\begin{proof}[Proof of \cref{thm:row regular}]
Consider the event
\[
\cE:=\Big\{\exists \bm{v}\in \bS^{n-1} \enskip \text{such that } \norm{Q_n\bm{v}} \le \eps\frac{\rho}{\sqrt{n}}\Big\}.	
\]
	Fix any realization of the matrix $Q_n$ such that the event holds, i.e. there exists a vector $\bm{v}\in \bS^{n-1}$ with $\norm{Q_n\bm{v}} \le \frac{\eps}{\sqrt{n}}$. Since $Q_n$ and its transpose have the same singular values, there is a vector $\bm{x} \in \bS^{n-1}$ such that $\norm{\bm{x}^{\intercal}Q_n} \le \frac{\eps}{\sqrt{n}}$. From this it follows that
	\begin{align*}
	\Pr(\cE) &\le \Pr\Big\{\inf_{\bm{x}\in \Comp(\de,\rho)} \norm{\bm{x}^{\intercal}Q_n}\le \eps \frac{\rho}{\sqrt{n}}\Big\}+\Pr\Big\{\inf_{\bm{x}\in \Incomp(\de,\rho)} \norm{\bm{x}^{\intercal}Q_n}\le \eps \frac{\rho}{\sqrt{n}}\Big\}\\
	&\le \Pr\Big\{\inf_{\bm{x}\in \Comp(\de,\rho)} \norm{\bm{x}^{\intercal}Q_n}\le \sqrt{n}/270\Big\}+\frac{1}{\de}\Pr\big\{\dist(R_n,H_n)\le \eps\big\}\\
	&\le 2e^{-c_{10}n}+\frac{1}{\de}(C\eps +2 e^{-cn}),
	\end{align*}
	where the second line follows from \cref{invertibility via distance}, and in the last passage \cref{anticoncentration estimate} and \cref{thm:distances} were used.
	 This completes our proof.
\end{proof}

The remainder of this section is devoted to a proof of \cref{anticoncentration estimate}. We recall a special case of Theorem 4 from \cite{Livshyts18}.

\begin{thm}[Sharp net for deterministic matrices]\label{sharp net}
Consider any $S\subset \bS^{n-1}$. Pick any $\al \in (0,\frac12)$, $\beta \in (0,\frac{\al}{10})$. Let $n\ge 1/\al^2$. There exists a (deterministic) net $\cN \subset S+\frac{4\beta}{\al}B_2^n$ with 
\[
\norm{\cN} \le N(S,\be B_2^n)\cdot e^{C_{11}\al^{0.08}\log(1/\alpha)n}
\]
such that for every $m\in \bN$ and for every (deterministic) $m\times n$ matrix $A$, the following holds: for every $\bm{x} \in S$ there exists $\bm{y}\in \cN$ satisfying
\[
\norm{(\bm{x}-\bm{y})^{\intercal}A} \le \frac{2\be}{\al \sqrt{n}} \lVert A\rVert_{\HS}.
\]
Here $C_{11}>0$ is an absolute constant.
\end{thm}


We also embrace the following anti-concentration estimate due to Jain, Sah and Sawhney (private communication), which answered Question 4.1 in the previous version of this manuscript.

\begin{lem}
	\label{lem:anti-conc}
	For each $\bm{x}\in \bS^{n-1}$, we have
	\[
	\Pr\big\{|\bm{x}^{\intercal}Q_n|\le \sqrt{n}/90\big\}\le e^{-n/3000}.
	\]
\end{lem}

The proof of \cref{lem:anti-conc} will be given in the Appendix. We conclude this section with a proof of \cref{anticoncentration estimate}, employing \cref{lem:anti-conc}.

\begin{proof}[Proof of \cref{anticoncentration estimate}]
Set $\al=540(\de+\rho)$ and $\be=\de +\rho$, where $0<\de,\rho \ll 1$. Assume that $n \ge 1/\al^2$.
Observe that
\[
N(\Comp(\de,\rho),\be B_2^n) \le \binom{n}{\de n}\cdot \left(\frac{4}{\de}\right)^{\de n} \le e^{2\de \log(4/\de)n}.
\]
Applying \cref{sharp net} together with the above observation, we get a net $\cN \subset \frac32 B_2^n \setminus \frac12 B_2^n$ with
\[
\norm{\cN} \le  e^{2\de \log(4/\de)n} \cdot e^{C_{11}\al^{0.08}\log(1/\alpha)n} \le e^{C\al^{0.08}\log(1/\al)n}.
\]
For fixed $\bm{y}\in \cN$, \cref{lem:anti-conc} tells us that
\[
\Pr\big\{|\bm{y}^{\intercal}Q_n|\le \sqrt{n}/135\big\}\le e^{-n/3000}
\]
Taking the union bound, we then get
\[
\Pr\Big\{\inf_{\bm{y}\in \cN}|\bm{y}^{\intercal}Q_n|\le \sqrt{n}/135\Big\} \le e^{C\al^{0.08}\log(1/\al)n}\cdot e^{-n/3000}< e^{-n/4000}.
\]

Consider the event
\[
\cE:=\Big\{\inf_{\bm{x}\in \Comp(\de,\rho)} \norm{\bm{x}^{\intercal}Q_n}\le \sqrt{n}/270\Big\}.
\]
To bound $\Pr(\cE)$, we suppose that $\cE$ occurs. Then $\norm{\bm{x}^{\intercal}Q_n}\le \sqrt{n}/270$ for some $\bm{x}\in \bS^{n-1}$. Since $\lVert Q_n\rVert_{\HS} \le n$, \cref{sharp net} shows the existence of $\bm{y}\in \cN$ with
\[
\norm{(\bm{x}-\bm{y})^{\intercal}Q_n} \le (2\be/\al)\sqrt{n} \le \sqrt{n}/270. 
\]
In particular, one has
\[
\norm{\bm{y}^{\intercal}Q_n} \le \norm{\bm{x}^{\intercal}Q_n}+\norm{(\bm{x}-\bm{y})^{\intercal}Q_n} \le \sqrt{n}/135.
\]
Therefore, we find
\[
\Pr(\cE) \le \Pr\big\{|\bm{y}^{\intercal}Q_n|\le \sqrt{n}/135\big\}\le e^{-n/4000}
\]
for $n$ sufficiently large, which completes our proof. 
\end{proof}

\section{Anti-concentration for combinatorial statistics}
\label{sec:small ball probability}
In this section we will derive \cref{binomial} from a more general result, namely \cref{thm:L-O}. We will prove \cref{binomial} in \cref{sec:general distributions to binomial} assuming the validity of \cref{thm:L-O}. We justify \cref{thm:L-O} in \cref{sec:anti-concentration}.

Let $\bm{a}$ and $\bm{v}$ be two vectors in $\bR^n$. The combinatorial statistic
\[
W_{\bm{a},\bm{v}}:=a_1 v_{\sigma(1)}+\ldots+a_n v_{\sigma(n)},
\]
where $\sigma$ is a uniformly random permutation of $[n]$, plays a fundamental role in statistics (see the book \cite{Good06} for an overview) as well as probability (see e.g. \cite{Cook17,FJLS20,Jain20,KST19,LLTTY17,LLTTY19,Nguyen13,NV13}). 
In analogy with the least common denominator (LCD) developed by Rudelson and Vershynin \cite{RV08}, we define a combinatorial version of LCD, which will be instrumental in controlling the anti-concentration of $W_{\bm{a},\bm{v}}$.

\begin{defn}[Combinatorial least common denominator]
Given two vectors $\bm{a}$ and $\bm{v}$ in $\bR^n$, as well as parameters $L,u>0$, the Combinatorial Least Common Denominator of the pair $(\bm{a},\bm{v})$ is
\[
\CLCD_{L,u}^{\bm{a}}(\bm{v}):=\inf\Big\{\theta>0\colon \dist(\theta\cdot \D(\bm{a})\otimes \D(\bm{v}),\bZ^{\binom{n}{2}^2})<\min\Big(u\norm{\D(\bm{a})\otimes \D(\bm{v})}, L\Big)\Big\}.
\]
Here by $\otimes$ we denote the tensor product.\footnote{In particular, $\D(\bm{a})\otimes \D(\bm{v})$ is a vector in $\bR^{\binom{n}{2}}$ whose $(i,j,k,\ell)$-coordinate is $(a_i-a_j)(v_k-v_{\ell})$, for $1\le i <j\le n$ and $1\le k<\ell\le n$.}
\end{defn}

The usefulness of $\CLCD$ is demonstrated in the following result, which shows how $\CLCD$ of the pair $(\bm{a},\bm{v})$ governs the small ball probability of $W_{\bm{a},\bm{v}}$.

\begin{thm}[Small ball probability]\label{thm:L-O}
Let $\bm{a}$ and $\bm{v}$ be two vectors in $\bR^n$ with $\norm{\D(\bm{a})\otimes \D(\bm{v})} \ge b n^{3/2}$ for some $b>0$. Let $L>0$ and $u\in (0,1)$. Then for any $\eps \ge 0$, we have
\[
\cL(W_{\bm{a},\bm{v}},\eps) \le C\eps +\frac{C}{\CLCD_{L,u}^{\bm{a}}(\bm{v})}+C e^{-8L^2/n^3}.
\]
The constant $C>0$ here depends only on $b$ and $u$.	
\end{thm}

\subsection{Deriving \cref{binomial} from \cref{thm:L-O}}
\label{sec:general distributions to binomial}

In this section we formally derive \cref{binomial} from \cref{thm:L-O}. For the reader's convenience, we restate \cref{binomial}.

\binomial*

\begin{proof}[Proof of \cref{binomial} assuming \cref{thm:L-O}]
	Let $\bm{a}:=(\underbrace{1,\ldots,1}_\text{$n/2$},0,\ldots,0) \in \{0,1\}^n$, $L:=\al n/2$ and $u:=\ga$. We can interpret $\D(\bm{a})\otimes\D(\bm{v})$ as a collection of $n^2/4$ copies of $\D(\bm{v})$. Thus we have that
	\[
	\norm{\D(\bm{a})\otimes \D(\bm{v})}=\tfrac12 n\norm{\D(\bm{v})} \ge \tfrac12 b n^{3/2},
	\]
	and that
	\[
	\CLCD^{\bm{a}}_{L,u}(\bm{v})=\CLCD_{\al,\ga}(\bm{v}).
	\]
	Hence \cref{thm:L-O} is applicable. Noting that $W_{\bm{a},\bm{v}}$ and $W_{\bm{v}}$ have the same law, we then get
	\[
	\cL(W_{\bm{v}},\eps)=\cL(W_{\bm{a},\bm{v}},\eps) \le C\eps +\frac{C}{\CLCD_{\al,\ga}(\bm{v})} + C e^{-2\al^2/n}. \qedhere
	\]
\end{proof}

\subsection{Proof of \cref{thm:L-O}}
\label{sec:anti-concentration}

The proof is closely modeled after \cite[Theorem 4.1]{RV08}.
We first recall the following anti-concentration inequality due to Ess\'een (see e.g. \cite{Esseen, RV08}).

\begin{lem} \label{Esseen} 
	Given a random variable $\xi$ with the characteristic function $\varphi(\cdot)=\bE  \exp(2\bm{\pi} \bm{i}\xi\cdot)$, one has
	\begin{equation*}\label{eqn:esseen} 
	\cL(\xi,\eps)\lesssim \int_{-1}^{1}\norm{\varphi\left(\frac{\theta}{\eps}\right)}\,d\theta, \quad \eps \ge 0.
	\end{equation*}	
\end{lem} 

To use Ess\'een's lemma, we require the following critical estimate for the characteristic function of the statistic $W_{\bm{a},\bm{v}}$.

\begin{thm}[{Roos \cite[Theorem 1.1]{Roos19}}]\label{characteristic function}
	Let $\bm{a}=(a_1,\ldots,a_n)$ and $\bm{v}=(v_1,\ldots,v_n)$ be vectors in $\bR^n$. Let $\vp$ be the characteristic function of $W_{\bm{a},\bm{v}}$. Then 
	\begin{equation*}
	|\vp(\theta)| \le \Bigg[\frac{1}{\binom{n}{2}^2}\sum_{\{s,t\},\{p,q\}\in \binom{[n]}{2}} \cos^2\bm{\pi}\theta (a_{s}-a_{t})(v_{p}-v_{q})\Bigg]^{(n-1)/4}, \quad \theta \in \bR.
	\end{equation*}
\end{thm}

We are now in a position to prove \cref{thm:L-O}.

\begin{proof}[Proof of \cref{thm:L-O}]
	Take any $\eps\ge 1/\CLCD^{\bm{a}}_{L,u}(\bm{v})$. From \cref{characteristic function} we know that
	\[
	\norm{\varphi\Bigg(\frac{\theta}{\eps}\Bigg)}\le \Bigg[\frac{1}{\binom{n}{2}^2} \sum\limits_{\{s,t\},\{p,q\}\in \binom{[n]}{2}}\cos^2\Big(\frac{ \bm{\pi}\theta(a_s-a_t)(v_p-v_q)}{\eps}\Big)\Bigg]^{(n-1)/4}.
	\]
	
	By convexity, we have that $|\sin \bm{\pi} z| \ge 2 \dist(z,\bZ)$ for any $z\in\bR$. Thus
	\begin{equation*}
	\cos^2 \bm{\pi} z =1-\sin^2 \bm{\pi} z \le 1-4 \dist^2(z,\bZ).
	\end{equation*}
	
	It follows that
	\begin{align*}
	\norm{\varphi\left(\frac{\theta}{\eps}\right)} &\le \Bigg[1-\frac{4 }{\binom{n}{2}^2} \sum_{\{s,t\},\{p,q\}\in \binom{[n]}{2}}\dist^2\Big(\frac{\theta}{\eps}\, (a_s-a_t)(v_p-v_q),\bZ\Big)\Bigg]^{(n-1)/4}\\
	&\le \exp\Bigg(- \frac{8}{n^3}\sum_{\{s,t\},\{p,q\}\in \binom{[n]}{2}}\dist^2\Big(\frac{\theta}{\eps}\, (a_s-a_t)(v_p-v_q),\bZ\Big)\Bigg)\\
	&= \exp\left(-\frac{8}{n^3}\ \dist^2\Big(\frac{\theta}{\eps}\, \D(\bm{a})\otimes \D(\bm{v}),\bZ^{\binom{n}{2}^2
	}\Big)\right),
	\end{align*}
	where in the second inequality we used the fact that $1-z\le e^{-z}$ for any $z\in \bR$.
	
	Combining this with \cref{Esseen} gives
	\begin{align*}
	\notag \cL(W_{\bm{a},\bm{v}},\eps) &\lesssim \int_{-1}^{1}\norm{\varphi\left(\frac{\theta}{\eps}\right)}\,d\theta\\
	&\lesssim \int_{-1}^{1}\exp\left(-8 h^2(\theta)/n^3\right)\,d\theta,
	\end{align*}
	where $h(\theta):=\dist(\frac{\theta}{\eps}\cdot \D(\bm{a})\otimes \D(\bm{v}),\bZ^{\binom{n}{2}^2
	})$.
	
	Since $1/\eps \le \CLCD^{\bm{a}}_{L,u}(\bm{v})$, it follows that for any $\theta \in [-1,1]$ we have
	\[
	h(\theta) \ge \min\Big(u\Big|\frac{\theta}{\eps}\cdot\D(\bm{a})\otimes\D(\bm{v})\Big|,L\Big)\ge \min\Big(\frac{ubn^{3/2}\norm{\theta}}{\eps},L\Big),
	\]
	assuming that $\norm{\D(\bm{a})\otimes \D(\bm{v})} \ge b n^{3/2}$. Therefore,
	\begin{align*}
	\cL(W_{\bm{a},\bm{v}},\eps) &\lesssim \int_{-1}^{1} \left[\exp\left(-8(ub\theta/\eps)^2\right)+\exp(-8L^2/n^3)\right]\,d\theta\\
	& \lesssim \frac{\eps}{ub}+\exp(-8L^2/n^3). \qedhere
	\end{align*}
\end{proof}

\section{Concluding remarks}
\label{sec:concluding remarks}

In this section we highlight some possible avenues for further investigation.

\subsection{Exchangeable random matrices.}
Let $M_n$ be a random $n\times n$ matrix. One source of motivation for finding good lower bounds on the least singular value $s_n(M_n)$ is its relation to the problem of proving the circular law for the distribution of eigenvalues of $M_n$. A model of random matrices, which is most relevant to us, was introduced in \cite{ACW16}.

Let $(a_{ij})_{1\le i,j \le n}$ be a deterministic real matrix such that 
\[
\sum_{i,j=1}^{n}a_{ij}=0 \enskip \text{and} \enskip \sum_{i,j=1}^n a_{ij}^2=n^2.
\]
We consider the exchangeable random matrix $M_n$ obtained by shuffling the deterministic matrix $(a_{ij})$ using a random uniform permutation, i.e.,
\[
M_n=(a_{\sigma(i,j)})_{1\le i \le j \le n}, \enskip \text{where} \enskip \sigma \enskip \text{is a uniformly random permutation of the set} \enskip \{(i,j)\colon 1\le i,j \le n\}.
\] 
Such random matrices have dependent entries, dependent rows, and dependent columns. In order to prove the circular law for $M_n$, Adamczak, Chafai and Woff \cite[Theorem 1.1]{ACW16} established a polynomial bound on the smallest singular value of the shifted matrix $\frac{1}{\sqrt{n}}M_n-z I$:
\[
\Pr\Big\{s_n\Big(\frac{1}{\sqrt{n}}M_n-zI\Big)\le \frac{\eps}{\sqrt{n}}\Big\} \lesssim_{K,z} \eps + \frac{1}{\sqrt{n}} \quad \text{for every } z\in \bC \enskip \text{and } \eps \ge 0,
\]
where $K:=\max_{i, j}|a_{ij}|$. They asked whether the factor $1/\sqrt{n}$ in the above estimate can be improved to $e^{-cn}$. We believe that our techniques (combined with additional twists) should allow one to solve this problem.

\subsection{Inverse Littlewood-Offord theory}
Given two vectors $\bm{a}$ and $\bm{v}$ in $\bR^{n}$, we define the {\em concentration probability} as

\[
\rho_{\bm{a},\bm{v}}:=\sup_{x}\Pr(W_{\bm{a},\bm{v}}=x).
\]

S{\"o}ze \cite[Corollary 5]{Soze17} showed $\rho_{\bm{a},\bm{v}} \lesssim \frac{1}{n}$ assuming that $\bm{a}=(1,2,\ldots,n)$ and that $\bm{v}$ is a non-constant vector, and used this estimate to bound the expected number of real roots of random polynomials with exchangeable coefficients. Motivated by their study of representations of reductive groups, Huang, McKinnon and Satriano \cite{HMS20} recently raised the problem of bounding $\rho_{\bm{a},\bm{v}}$ when $\bm{a}$ has distinct coordinates and $\bm{v}$ is a non-constant vector. Under these assumptions they showed that $\rho_{\bm{a},\bm{v}} \lesssim \frac{1}{n}$, which generalizes S{\"o}ze's result. Shortly after the Huang-McKinnon-Satriano paper, Pawlowski \cite{Paw20} gave a combinatorial proof of the refinement

\begin{equation}\label{eq:Pawlowski}
\rho_{\bm{a},\bm{v}} \le \frac{2\lfloor n/2\rfloor}{n(n-1)}.
\end{equation}

This estimate is sharp, as demonstrated by $\bm{a}=(1,\ldots,n)$ and $\bm{v}=(-\sum\limits_{i=2}^{n-1}i,-\sum\limits_{i=2}^{n-1}i, n+1,\ldots, n+1)$.
It is likely that \eqref{eq:Pawlowski} can be improved significantly by making additional assumptions about $\bm{a}$ and $\bm{v}$. Phrased differently, it would be very interesting to find an answer to the following basic question:

\begin{ques}\label{ques2}
What is the underlying reason why $\rho_{\bm{a},\bm{v}}$ could be large?
\end{ques}

We remark that Nguyen and Vu \cite[Theorem 4.4]{NV13} gave a partial answer to the above question when $\bm{a}=(\underbrace{1,\ldots,1}_\text{$n/2$},0,\ldots,0) \in \{0,1\}^n$.

Another interesting direction is to extend Pawlowski's result to other ambient groups. For a detailed account of the Inverse Littlewood-Offord theory, we refer the reader to an excellent survey of Nguyen and Vu \cite{NV13-survey}.

\section*{Acknowledgement}
We thank Vishesh Jain, Ashwin Sah and Mehtaab Sawhney for their assistance with the proof of \cref{lem:anti-conc}, and Prof. Bero Roos for pointing us to the reference \cite{Roos19}.

\appendix
\section{(by Jain, Sah and Sawhney): Proof of \cref{lem:anti-conc}}

We first recall the following variant of Paley-Zygmund inequality given by Litvak, Pajor, Rudelson and Tomczak-Jaegermann \cite[Lemma 3.5]{LPRT05}.

\begin{lem}\label{lem:paley-zygmund}
Let $X$ be a nonnegative random variable with $\bE[X^4]<\infty$. Then for $0\le \lambda <\sqrt{\bE[X^2]}$ we have 
	\[
	\Pr\big\{X>\lambda\big\}\ge \frac{(\bE[X^2]-\lambda^2)^2}{\bE[X^{4}]}.\]
\end{lem}

We also need the following general hypercontractive estimate, which is a variant of \cite[Theorem~10.21]{Donnell14}.

\begin{lem}\label{lem:moment-hypercontractivity}
	Let $X$ be a vector of independent Bernoulli random variables with minimum atom probability lower bounded by $b$. Let $f$ be a polynomial function of $X$ with degree at most $d$. Then, for all $q\ge 2$ we have that 
	\[
	\bE[|f(X)|^q]^{1/q}\le (\sqrt{q-1}\,b^{1/q-1/2})^d\bE[f(X)^2]^{1/2}.
	\]
\end{lem}

We now prove \cref{lem:anti-conc}; it is closely modelled after \cite[Proposition~3.4]{LPRT05}.
\begin{proof}[Proof of \cref{lem:anti-conc}]
	Let $\Gamma$ be the first $N = n/4$ columns of $Q_n$. Then, $\Gamma$ is an $n\times N$ matrix with independent rows, with each row distributed as the indicator of the elements in $[N]$ for a randomly chosen size $n/2$ subset of $[n]$. We also have $\norm{\bm{x}} = 1$. Let the elements of $\Gamma$ be $\Gamma_{ij}$, and the columns of $\Gamma$ be $\Gamma_j$. Let
	\[
	y_j = (\bm{x}^{\intercal}\Gamma)_j = \bm{x}^{\intercal}\Gamma_j = \sum_{i=1}^n x_i\Gamma_{ij}.
	\]
	For every $\tau>0$, we have
	\begin{align*}
	\Pr\big\{\norm{\bm{x}^{\intercal}Q_n}\le t\sqrt{N}\big\}&\le\Pr\big\{\norm{\bm{x}^{\intercal}\Gamma}\le t\sqrt{N}\big\} =\Pr\Big\{\sum_{j=1}^Ny_j^2\le t^2N\Big\} = \Pr\Big \{N-\frac{1}{t^2}\sum_{j=1}^N y_j^2\ge 0\Big\}\\
	&\le\bE\exp\bigg(\tau N-\frac{\tau}{t^2}\sum_{j=1}^Ny_j^2\bigg)\le e^{\tau N}\prod_{j=1}^N\sup_{\Gamma_1,\ldots,\Gamma_{j-1}}\bE[\exp(-\tau y_j^2/t^2)|\Gamma_1,\ldots,\Gamma_{j-1}],
	\end{align*}
	where the last inequality is obtained by iterating the law of total expectation and since all terms considered are positive.
	
	Now, we study $y_j$ conditional on $\Gamma_1,\ldots,\Gamma_{j-1}$. Given these values, we have $\bE[y_j^2]\ge \Var[y_j]\ge \frac{2}{9}$ since conditional on the first $j-1$ columns for $j\le n/4$, each entry in the $j^{th}$ column is distributed as $\Ber(p)$ with $p$ between $1/3$ and $2/3$. Next, using hypercontractivity (\cref{lem:moment-hypercontractivity}) with $q=4,b=1/3$ and $d=1$, we find that
	\[
	\bE[y_j^4]\le 27\bE[y_j^2]^2.
	\]
	Thus, setting $\lambda^2 = \bE[y_j^2]/2$ in \cref{lem:paley-zygmund}, we find that
	\begin{align*}
	\Pr\big\{y_j>1/3\big\}&\ge \Pr\big\{y_j^2>\bE[y_j^2]/2\big\}\\
	&\ge \frac{\bE[y_j^2]^2}{4\bE[y_j^4]}\ge\frac{1}{108}.
	\end{align*}
	Using this we immediately find that
	\begin{align*}
	\bE[\exp(-\tau y_j^2/t^2)|\Gamma_1,\ldots,\Gamma_{j-1}] &\le 107/108 + 1/108\exp(-\tau/(9t^2)).
	\end{align*}
	Therefore,
	\[\Pr\big\{\norm{\bm{x}^{\intercal}Q_n}\le t\sqrt{N}\big\}\le e^{\tau N}(107/108+1/108\exp(-\tau/(9t^2)))^N,\]
	and setting $t = 1/45$ and $\tau = 1/500$ gives
	\[\Pr\big\{\norm{\bm{x}^{\intercal}Q_n}\le\sqrt{n}/90\big\}\le\exp(-n/3000).\qedhere\]
\end{proof}

\end{document}